\documentclass[11pt,reqno]{amsart}
\usepackage[margin=3.5cm]{geometry}
\usepackage{mlmodern}

\usepackage{amsmath}
\usepackage{amssymb}
\usepackage{mathrsfs}
\usepackage{mathtools}

\usepackage{amsthm}

\theoremstyle{plain}

\newtheorem{mainthm}{Theorem}

\swapnumbers

\newtheorem{theorem}[subsection]{Theorem}
\newtheorem{proposition}[subsection]{Proposition}
\newtheorem{lemma}[subsection]{Lemma}
\newtheorem{corollary}[subsection]{Corollary}

\theoremstyle{definition}
\newtheorem{example}[subsection]{Example}
\newtheorem{definition}[subsection]{Definition}

\theoremstyle{remark}
\newtheorem{remark}[subsection]{Remark}

\numberwithin{equation}{subsection}

\usepackage[title]{appendix}

\usepackage{hyperref}
\usepackage{tikz-cd}
\usepackage{theoremref}
\usepackage{enumitem}

\usepackage{comment}
\usepackage{xcolor}

\usepackage{algorithm}
\usepackage{algpseudocode}
\usepackage{float}

\usepackage{multirow}
\usepackage{bm}

\usepackage{jlcode}

\usepackage{subcaption}
\usepackage{caption}

\usepackage{etoolbox}
\apptocmd{\thebibliography}{\raggedright}{}{}

\newcommand*{\bbB}{\mathbb{B}}

\newcommand*{\bbD}{\mathbb{D}}
\newcommand*{\bbN}{\mathbb{N}}
\newcommand*{\bbP}{\mathbb{P}}

\newcommand*{\bbR}{\mathbb{R}}
\newcommand*{\bbS}{\mathbb{S}}
\newcommand*{\bbZ}{\mathbb{Z}}

\newcommand*{\calC}{\mathcal{C}}
\newcommand*{\calD}{\mathcal{D}}
\newcommand*{\calE}{\mathcal{E}}
\newcommand*{\calF}{\mathcal{F}}

\newcommand*{\calL}{\mathcal{L}}

\newcommand*{\calP}{\mathcal{P}}

\newcommand*{\calS}{\mathcal{S}}
\newcommand*{\calT}{\mathcal{T}}
\newcommand*{\calU}{\mathcal{U}}
\newcommand*{\calV}{\mathcal{V}}
\newcommand*{\calX}{\mathcal{X}}
\newcommand*{\calW}{\mathcal{W}}

\newcommand*{\id}{\mathrm{id}}

\DeclareMathOperator{\codim}{codim}

\DeclareMathOperator{\interior}{int}
\DeclareMathOperator{\gr}{gr}
\DeclareMathOperator{\fl}{fl}

\DeclareMathOperator{\conv}{conv}

\newcommand{\qtA}{%
\begin{tikzpicture}[baseline=-0.6ex,scale=.35,line width=.5pt]
  \draw[fill=white] (0,0) circle (5pt);
\end{tikzpicture}}

\newcommand{\qtB}{%
\begin{tikzpicture}[baseline=-0.6ex,scale=.35,line width=.5pt]
  \fill (0,1) circle (5pt);
  \draw (0,1) -- (0,0);
  \draw[fill=white] (0,0) circle (5pt);
\end{tikzpicture}}

\newcommand{\qtC}{%
\begin{tikzpicture}[baseline=-0.6ex,scale=.35,line width=.5pt]
  \fill (-.8,1) circle (5pt);
  \fill (.8,1) circle (5pt);
  \draw (-.8,1) -- (0,0);
  \draw (.8,1) -- (0,0);
  \draw[fill=white] (0,0) circle (5pt);
\end{tikzpicture}}

\newcommand{\qtD}{%
\begin{tikzpicture}[baseline=-0.6ex,scale=.35,line width=.5pt]
  \fill (-1,1) circle (5pt);
  \fill (0,1) circle (5pt);
  \fill (1,1) circle (5pt);
  \draw (-1,1) -- (0,0);
  \draw (0,1) -- (0,0);
  \draw (1,1) -- (0,0);
  \draw[fill=white] (0,0) circle (5pt);
\end{tikzpicture}}

\newcommand{\qtE}{%
\begin{tikzpicture}[baseline=-0.6ex,scale=.35,line width=.5pt]
  \fill (-1,1) circle (5pt);
  \fill (-.3,1) circle (5pt);
  \fill (.4,1) circle (5pt);
  \fill (1.1,1) circle (5pt);
  \draw (-1,1) -- (0,0);
  \draw (-.3,1) -- (0,0);
  \draw (.4,1) -- (0,0);
  \draw (1.1,1) -- (0,0);
  \draw[fill=white] (0,0) circle (5pt);
\end{tikzpicture}}

\newcommand{\qtF}{%
\begin{tikzpicture}[baseline=-0.6ex,scale=.35,line width=.5pt]
  \fill (0,0.75) circle (5pt);
  \fill (0,1.5) circle (5pt);
  \draw (0,1.5) -- (0,0.75);
  \draw (0,0.75) -- (0,0);
  \draw[fill=white] (0,0) circle (5pt);
\end{tikzpicture}}

\title{Hilbert's 16th problem for arrangements of curves on a surface}
\author{Giacomo Maletto}
\date{}

\begin{document}

\begin{abstract}
    We introduce a combinatorial structure $(n,W,T)$ encoding the topological type of a curve transverse to a fixed cellular arrangement of curves on a compact real surface, in terms of intersection numbers, Dyck words and rooted trees. We apply this formalism to analyze a natural generalization of Hilbert's 16th problem to arrangements of curves. We obtain a complete classification of arrangements of three lines and a cubic, and a partial classification of arrangements of three lines and a quartic. This is achieved using Bézout-type obstructions, Viro's patchworking and translations, and by developing the \texttt{Julia} library \texttt{NWT} to handle large databases of curves.
\end{abstract}

\maketitle

\section{Introduction}

The purpose of this work is to lay the foundations of the study of the topological types of arrangements of curves on real compact surfaces, and then initiate the classification of arrangements of algebraic curves on $\bbP^2(\bbR)$ (what we call \textit{Hilbert's 16th problem} \cite{hilbert}).

A \textbf{$k$-arrangement} of curves on a surface $\calS$ is a $k$-tuple $(\calC_1,\dots,\calC_k)$ of curves where all pairs meet transversely and no three have a common point. Two arrangements $(\calC_1,\dots,\calC_k)$, $(\calC'_1,\dots,\calC'_k)$ are \textbf{topologically equivalent} or have the same \textbf{topological type} if there exists a homeomorphism $\eta\colon\calS\to\calS$ such that $\eta(\calC_i)=\calC'_i$ for all $i$ (for $\calS=\bbP^2(\bbR)$, this coincides with the ambient isotopy type, see Remark \ref{rem:ambient}).

In the first part of the paper (Sections \ref{sec:cell}, \ref{sec:geocom}, \ref{sec:ops}) we analyze the problem in the topological category. We prove that topological types of arrangements $(\calC_1,\dots,\calC_k,\calD)$ where $(\calC_1,\dots,\calC_k)$ is \textbf{cellular} (Definition \ref{def:cellular}) and $\calD$ is any additional transversal curve are encoded by combinatorial objects $(n,W,T)$ called \textbf{combinatorial curves}, where $n$ are integers, $W$ Dyck words and $T$ rooted trees. More precisely:

\begin{mainthm}\label{thm:mainA}
    Fix a cellular arrangement $(\calC_1,\dots,\calC_k)$. Then there is a bijection between arrangements $(\calC_1,\dots,\calC_k,\calD)$ up to topological equivalence and combinatorial curves up to $i$-automorphism (Definition \ref{def:ccc}).
\end{mainthm}

The topology of cellular decompositions of surfaces is the subject of Section \ref{sec:cell}, the bridge between geometry and combinatorics is described in Section \ref{sec:geocom}, while Section \ref{sec:ops} is a technical part of the paper which deals with various operations on combinatorial curves useful in the following.

In the second part of the paper (Sections \ref{sec:obstructions}, \ref{sec:constructions}, \ref{sec:results}) we focus on $\calS=\bbP^2(\bbR)$ and study which classes of $k$-arrangements are $(d_1,\dots,d_k)$-realizable, i.e.\ they have a representative $(\calC_1,\dots,\calC_k)$ where each $\calC_i$ is a nonsingular algebraic curve of degree $d_i$. As in any other incarnation of Hilbert's 16th problem, we can split the problem into \textit{obstructions to realizability} (Section \ref{sec:obstructions}) and \textit{constructions of realizable cases} (Section \ref{sec:constructions}). Our approach is to frame the classification in combinatorial terms: concretely, we have developed the \texttt{Julia} \cite{Julia} library \jlinl{NWT} \cite{NWTrepo} to handle large databases of combinatorial curves, and every operation described in this paper has an implementation in this library. Using it, we focus in particular on $(1,1,1,3)$- and $(1,1,1,4)$-realizability, i.e.\ arrangements of three lines and a cubic or a quartic. The main results (Section \ref{sec:results}) are the following.

\begin{mainthm}\label{thm:mainB}
    There are exactly 119 topological types of arrangements of three unlabeled\footnote{Allowing for permutations of the three lines: see Example \ref{ex:unlabeled} for details.} lines and a cubic.
\end{mainthm}

For arrangements of three lines and a quartic we have partial results given the increasing complexity. We are able to show the following.

\begin{mainthm}\label{thm:mainC}
    There are exactly 619 topological types of arrangements of three unlabeled lines and a quartic consisting of a single oval.
\end{mainthm}

Moreover, our classifications shows that in total there are between 9426 and 17624 topological types of arrangements of three unlabeled lines and a quartic. If we restrict to \textbf{floatless} arrangements (those in which every oval of the quartic meets at least one line) we get more precise bounds: the number of such cases lies between 1834 and 1883.

\subsection{An example}

Let us look at any curve, for example
\[\calC=\{(x:y:z)\in\bbP^2(\bbR)\mid f(x,y,z)=0\}\]
with
\[f=x^4+x^2y^2+2xy^3-y^4-2x^3z+xy^2z-y^3z-3x^2z^2-2xyz^2+2y^2z^2+yz^3+z^4.\]
We can plot this curve in $\bbP^2(\bbR)$ together with the axes
\[\calL_x=\{x=0\},\quad\calL_y=\{y=0\},\quad\calL_z=\{z=0\},\]
for example by considering $\{(x,y,z)\in\bbS^2\mid f=0\}$ and projecting onto the $(z=0)$-axis, obtaining Figure \ref{fig:first-example}. If we modify the coefficients of $f$ we get a topologically equivalent arrangement (actually they are rigid isotopic), until either $f=0$ becomes singular or it stops meeting the axes transversely.

\begin{figure}
    \centering
    \begin{minipage}{.3\textwidth}
    \includegraphics[width=1\linewidth]{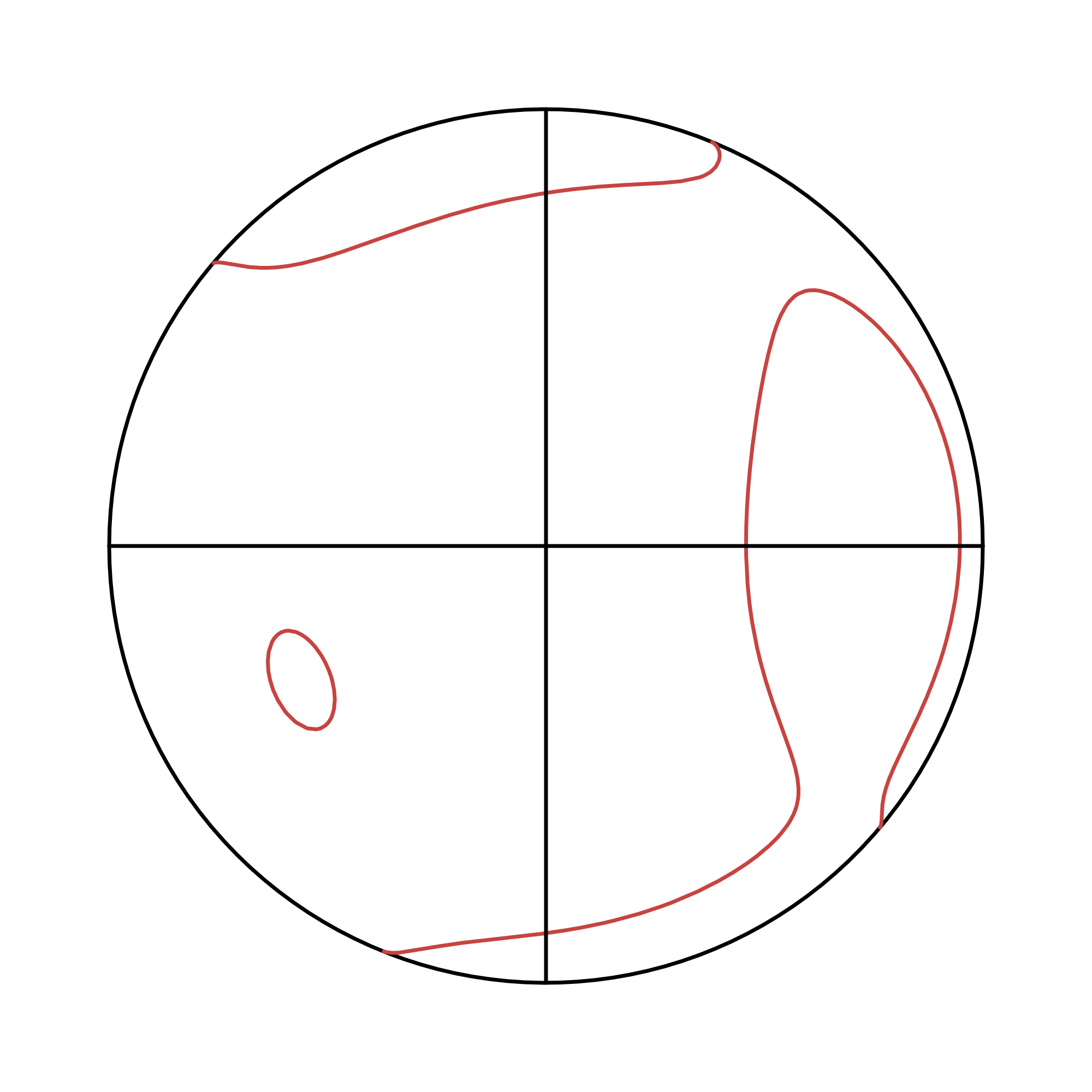}
    \end{minipage}%
    \begin{minipage}{.3\textwidth}
    \includegraphics[width=1\linewidth]{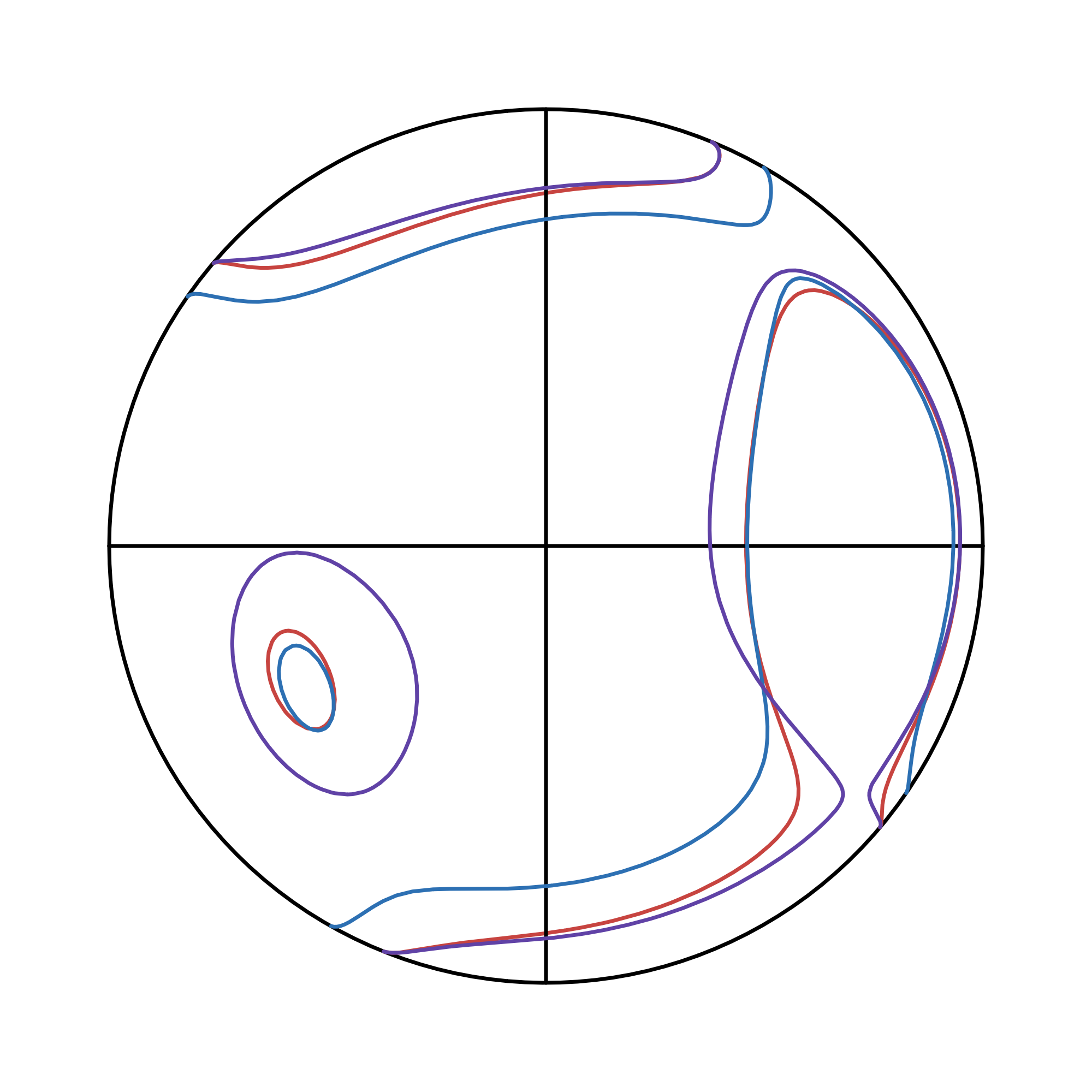}
    \end{minipage}
    \caption{The curve $f=0$ in $\bbP^2(\bbR)$ and some curves with the same topological type.}
    \label{fig:first-example}
\end{figure}

We would like to give a ``label'' to this arrangement, in the same way that the topological type of a single curve is given by a rooted tree and the information of whether it has a pseudoline or not [§\ref{subsubsec:original}]. A natural way to do so is the following: notice first that the base arrangement $(\calL_x,\calL_y,\calL_z)$ divides the ambient space $\bbP^2(\bbR)$ into a cellular complex, with 3 vertices, 6 edges and 4 faces. We name the edges $x,y,z,X,Y,Z$ and the faces $xyz,xYZ,XyZ,XYz$ depending on which edges constitute their boundary. Choose an orientation of the boundary of each face as in Figure \ref{fig:lines_xyz}.


\begin{figure}
    \centering
    \includegraphics[width=1\linewidth]{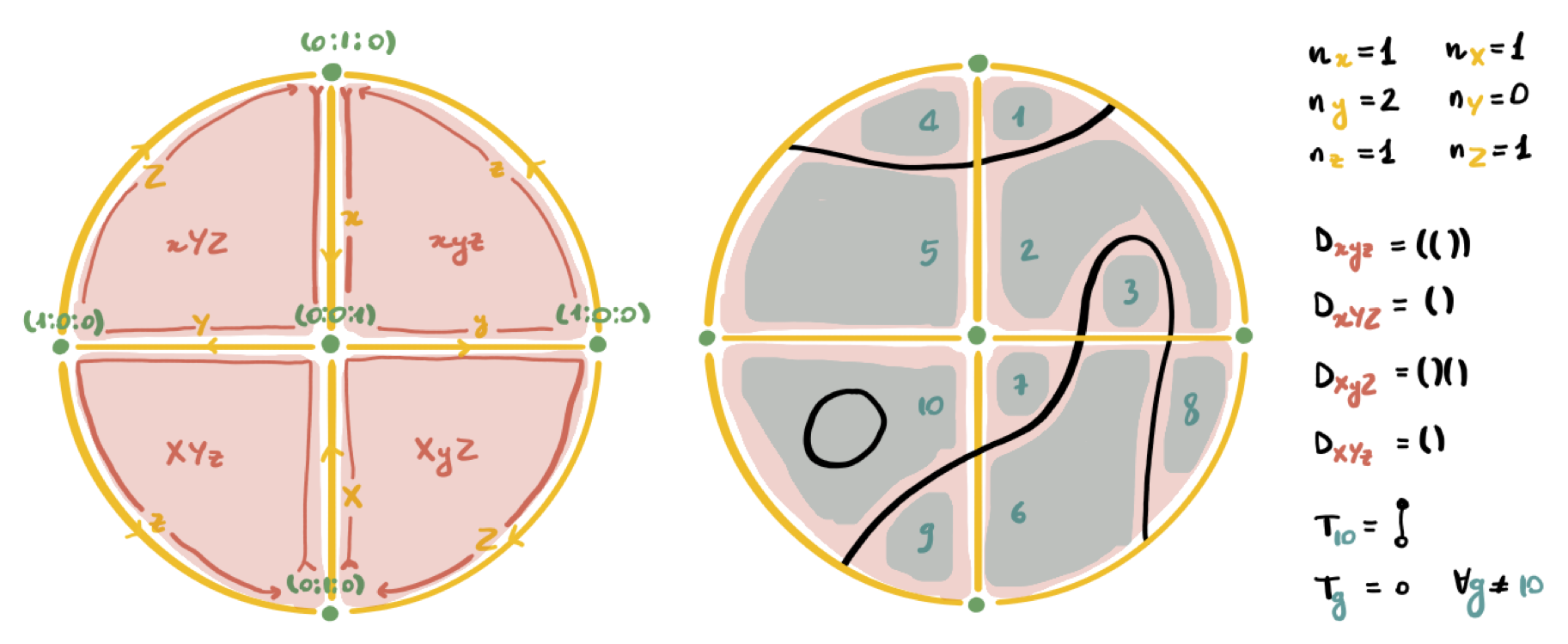}
    \caption{The cell complex induced by $(\calL_x,\calL_y,\calL_z)$ and the combinatorial curve $(n,W,T)$ induced by $(\calL_x,\calL_y,\calL_z,\calC)$.}
    \label{fig:lines_xyz}
\end{figure}

Next, we assign the following combinatorial information $(n,W,T)$ to the curve $\calC$:
\begin{itemize}
    \item for each edge $e$, $n_e=\#(e\cap\calC)$ is the number of intersections of $e$ with $\calC$;
    \item for each face $f$, we construct a Dyck word $W_f$ by starting with the empty word and while traveling along the boundary of $f$, adding a ``('' or a ``)'' to the end of the word whenever a component of $\calC$ is met resp.\ for the first or second time.
\end{itemize}
This takes care of the ovals of $\calC$ which meet the lines $\calL_x\cup\calL_y\cup\calL_z$. We call the union of these ovals the \textit{ground part} $\gr(\calC)$ of the curve, while the union of the components of $\calC$ not meeting any line is the \textit{floating part} $\fl(\calC)$. Now, the arrangement $(\calL_x,\calL_y,\calL_z,\gr(\calC))$ also forms a cellular complex, whose faces we call \textit{ground faces}. We complete the combinatorial description of $\calC$ with the following:
\begin{itemize}
    \item for each ground face $g$, a rooted tree whose vertices represent the connected components of $g\setminus\fl(\calC)$, with the root being the unique such component whose boundary is the boundary of $g$.
\end{itemize}

By Theorem \ref{thm:main}, the information contained in $(n,W,T)$ fully specifies the curve $\calC$ up to $\bbP^2(\bbR)$-homeomorphism fixing $(\calL_x,\calL_y,\calL_z)$. By Corollary \ref{cor:main_eq}, the orbit of $(n,W,T)$ through the $i$-automorphisms of $(\calL_x,\calL_y,\calL_z)$ fully specifies the topological type of the arrangement $(\calL_x,\calL_y,\calL_z,\calC)$.

The collection of data $(n,W,T)$ is called a \textit{combinatorial curve}, transversal to the combinatorial cell complex induced by the base arrangement $(\calL_x,\calL_y,\calL_z)$. In order to record and work with combinatorial curves, the \texttt{Julia} library \jlinl{NWT} is developed. For example, we can input the combinatorial curve we just saw as follows (see Remark \ref{rem:conventions} for an explanation of the conventions used):

\begin{jllisting}
nwt = NWT(lines_xyz, #combinatorial cell complex formed by 3 lines on P^2(R)
          [1, 2, 1, 1, 0, 1], #n
          [[1, 1, 0, 0], [1, 0], [1, 0, 1, 0], [1, 0]], #W, with 1="(", 0=")"
          [(10, [1, 0])]) #T, indicating the only non-trivial rooted tree
                          #is in region 10 with shape [1, 0]
\end{jllisting}

Among other things, from \jlinl{nwt} we can compute the topological type of $\calC$ in $\bbP^2(\bbR)$, whether the combinatorial curve respects certain criteria for realizability based on Bézout's theorem, and even print a picture of the curve as in Figure \ref{fig:first-example-love}. Every example appearing in this paper is also demonstrated computationally in the file \texttt{examples.ipynb}, which also shows how to use the database of curves we created.

\begin{figure}
    \centering
    \includegraphics[width=0.2\linewidth]{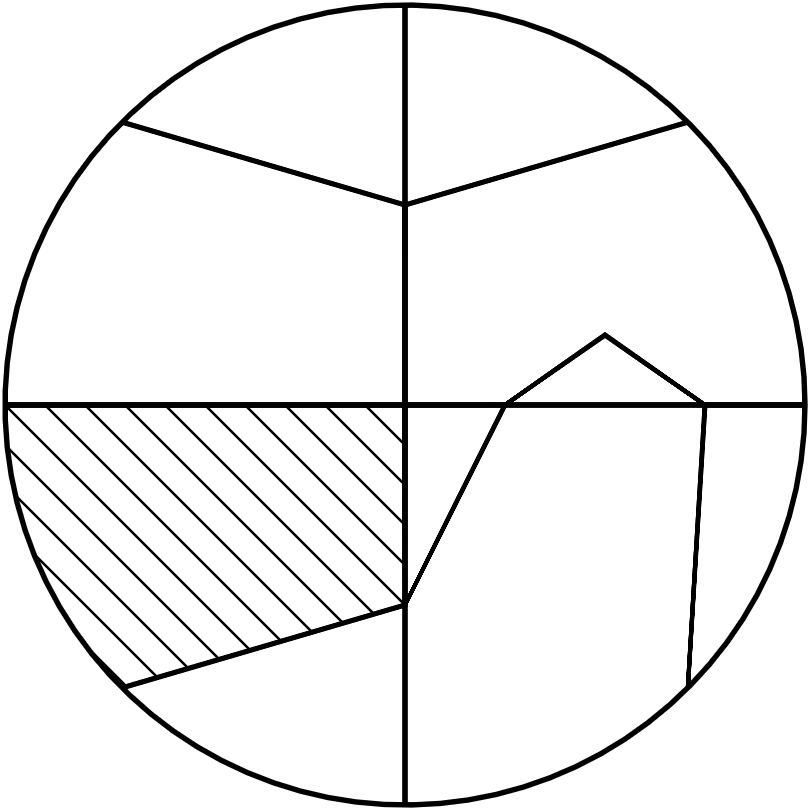}
    \caption{The combinatorial curve \jlinl{nwt} drawn with the library \jlinl{NWT}. The shading represents a region with a single floating oval.}
    \label{fig:first-example-love}
\end{figure}

\subsection{Motivation and history}

The problem we pose has already been studied successfully in many cases:

\subsubsection{$d=(d)$ in $\bbP^2(\bbR)$: Hilbert's original 16th problem}\label{subsubsec:original}

In its English translation published \cite{hilbert} in 1902, the first part of Hilbert's 16th problem is stated as follows:

\begin{quote}
    16. PROBLEM OF THE TOPOLOGY OF ALGEBRAIC CURVES AND SURFACES.
    
    The maximum number of closed and separate branches which a plane algebraic curve of the $d$-th order can have has been determined by Harnack. There arises the further question as to the relative position of the branches in the plane. [...] \textit{A thorough investigation of the relative position of the separate branches when their number is the maximum seems to me to be of very great interest, and not less so the corresponding investigation as to the number, form, and position of the sheets of an algebraic surface in space.}
\end{quote}

In our terms, the problem can be framed as follows: find out which classes of $1$-arrangements are $(d)$-realizable.

A nonsingular algebraic curve of degree $d$ (or more generally, a complete 1-dimensional topological submanifold) $\calC\subset\bbP^2(\bbR)$ admits a particularly easy description \cite{wilson1978hilbert}. Its connected components are all homeomorphic to $\bbS^1$. Since the fundamental group $\pi_1(\bbP^2(\bbR))=\bbZ/2\bbZ$ has two elements, we have two cases:
\begin{itemize}
    \item those connected components that are non-trivial as elements of $\pi_1(\bbP^2(\bbR))$ are called \textbf{pseudolines}, and there are none if $d$ is even and one if $d$ is odd;
    \item the connected components trivial in $\pi_1(\bbP^2(\bbR))$ are called \textbf{ovals}, and each of them separates $\bbP^2(\bbR)$ into two connected components, one of which is orientable and is called the interior of the oval.
\end{itemize}
A \textbf{region} is a connected component of the complement of $\calC$. Based on the observation that there is always a unique region whose closure is non-orientable (which is the complement of all interiors of ovals), we can associate to the curve $\calC$ a rooted tree, whose vertices are the regions, which are connected by an edge if the closures of the two corresponding regions intersect, and with root the unique non-orientable region. This tree, together with the binary information of whether the curve has a pseudoline or not, completely determines the topological type of the curve, and is commonly called its \textit{real scheme} \cite{rokhlin1978complex}. Our construction of combinatorial curves is a generalization of this phenomenon to larger arrangements.

\begin{figure}
    \centering
    \includegraphics[width=0.75\linewidth]{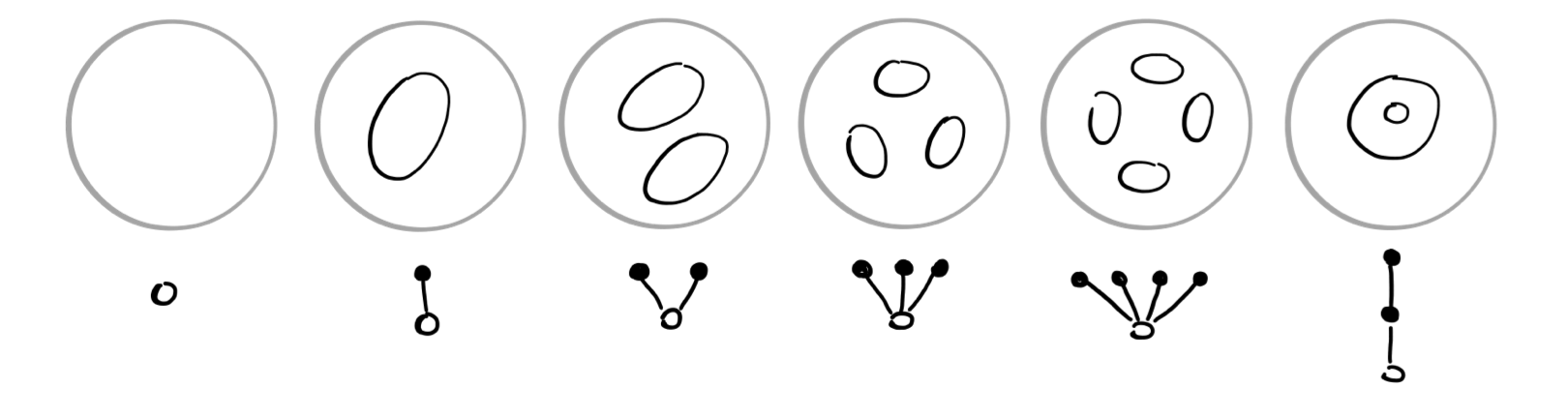}
    \caption{Topological types of quartic curves.}
    \label{fig:quartics}
\end{figure}

\begin{figure}
    \centering
    \includegraphics[width=0.9\linewidth]{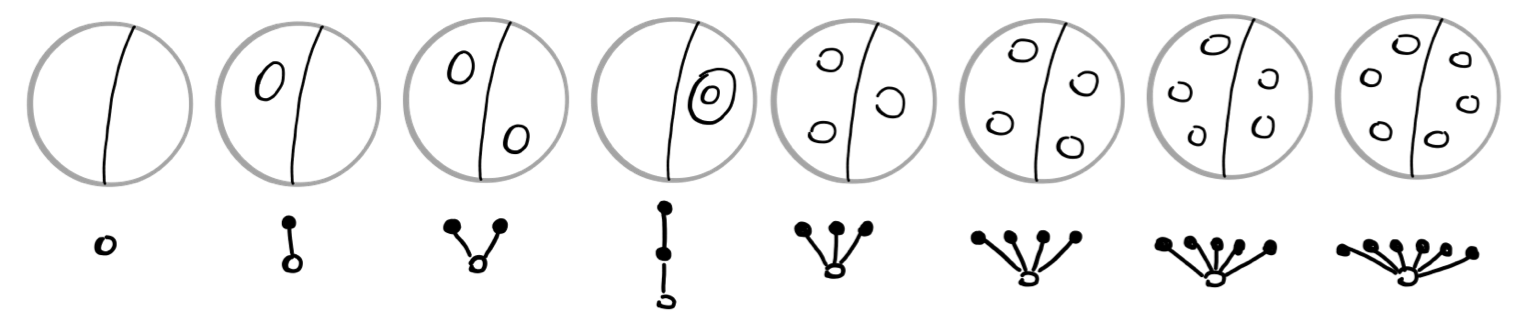}
    \caption{Topological types of quintic curves.}
    \label{fig:quintics}
\end{figure}

For example, for $d=4$ there are 6 topological types, namely the quartics consisting of 0,1,2,3 or 4 non-nested ovals, or 2 nested ovals \cite{plaumann2011quartic} (shown in Figure \ref{fig:quartics}).

Hilbert himself worked on the case $d=6$ of sextic curves on the plane, arriving ``by a complicated process'' to a classification of the possible topological types -- however, his arguments were incomplete, and only in 1969 the classification was completed by Gudkov, by fixing a mistake he himself made in 1954 \cite{viro2008sixteenth}.

The classification for $d=7$ was completed by Viro \cite{ya1980curves} in 1979, for $d=8$ it is still open but much is known \cite{orevkov2002classification} while for $d\ge9$ our picture is very incomplete.

\subsubsection{$d=(d)$ in other surfaces}

There are many works on Hilbert's 16th problem dealing with real curves on surfaces other than $\mathbb{P}^2(\bbR)$ \cite{ivanovich2025chambers}, for instance on quadrics \cite{degtyarev1999rigid,manzaroli2021real}, Hirzebruch and ruled surfaces \cite{zvonilov2003isotopies,degtyarev2011real} and real del Pezzo surfaces \cite{manzaroli2022real}. Additionally, Viro's Patchworking Construction \cite{viro2006patchworking} works for any toric surface.

We note that the techniques described in this article also work for any other nonsingular surface, but a similar classification for e.g.\ curves of a given bidegree $(d,e)$ on $\bbP^1(\bbR)\times\bbP^1(\bbR)$ or other toric varieties is relegated to future works.

\subsubsection{$d=(1,1,\dots,1)$: Folkman-Lawrence topological representation theorem}

The Folkman-Lawrence topological representation theorem \cite{folkman1978oriented} gives a correspondence between arrangements of pseudolines and rank 3 oriented matroids. The question of which of these classes are realizable from arrangements of lines becomes the question of which oriented matroids are stretchable, which is a NP-hard problem \cite{toth2017handbook}[§ 5.3].

Our article suggests another way to compute the possible pseudoline arrangements $(\calL_1,\dots,\calL_{k+1})$ given a fixed pseudoline arrangement $(\calL_1,\dots,\calL_k)$: these correspond exactly to the floatless combinatorial curves on $(\calL_1,\dots,\calL_k)$ whose topological type is that of a pseudoline, and respecting $(1,\dots,1)$-Bézout at order 0. The stretchability question is thus equivalent to the question of $(1,\dots,1)$-realizability of these configurations.

\subsubsection{Toric varieties: GKZ's refined Hilbert problem}

In \cite{gelfand1994discriminants}[§ 11.5] the ``refined Hilbert problem'' is introduced. Given a finite subset $A\in\bbZ^{k-1}$ and $Q=\conv(A)$ its convex hull, we denote by $\bbR^A$ the space of polynomials of the form $\sum_{\omega\in A}a_\omega x^\omega$ where $x^\omega=x_1^{\omega_1}\dots x_{k-1}^{\omega_{k-1}}$. Any non-zero $f\in\bbR^A$ naturally defines a hypersurface $Z_f$ inside the toric variety $X_A$. Suppose $X_A$ is smooth. The subset
\[\nabla_A(\bbR)=\{f\in\bbR^A\mid\text{$Z_f\subset X_A$ is singular}\}\]
forms an algebraic hypersurface in $\bbR^A$, called the \textit{$A$-discriminantal variety}, and the \textit{natural Hilbert problem} is the study of the connected components of the complement $\bbR^A\setminus\nabla_A(\bbR)$, corresponding to the study of the hypersurfaces $Z_f(\bbR)$ up to rigid isotopy in $X_A(\bbR)$. Letting
\[\bbR^A_{\text{gen}}=\{f\in\bbR^A\mid Z_f\pitchfork X^0(\Gamma)\text{ for all faces $\Gamma\subseteq Q$}\}\]
where $\pitchfork$ denotes transverse intersection (i.e.\ $X\pitchfork Y$ if $\codim(X\cap Y)=\codim(X)+\codim(Y)$ and $X\cap Y$ is smooth) and $X^0(\Gamma)$ is the orbit in $X_A$ associated to a face $\Gamma$, the \textit{refined Hilbert problem} is the study of the connected components of $\bbR^A_{\text{gen}}$ and corresponds to the study of arrangements formed by $Z_f(\bbR)$ and the orbits $X^0(\Gamma)$ intersecting transversely up to rigid isotopy.

For an integer $d>0$, if $A=\conv\{(0,0),(d,0),(0,d)\}\cap\bbZ^2$ then $X_A(\bbR)\simeq\bbP^2(\bbR)$ and the orbits $X^0(\Gamma)$ correspond to the points $(1:0:0),(0:1:0),(0:0:1)$, the axes $\{x=0\},\{y=0\},\{z=0\}$ and $\bbP^2$ itself. The refined Hilbert problem thus becomes the classification up to rigid isotopy of nonsingular projective algebraic curves of degree $d$ intersecting $\{xyz=0\}$ transversely. Thus, our analysis of $(1,1,1,d)$-realizability is a weakening of this problem to up to topological equivalence (which in $\bbP^2(\bbR)$ is equivalent to ambient isotopy) instead of rigid isotopy.

\subsection{Notation}
    The open and closed unit balls in $\bbR^n$ are denoted respectively $\bbB^n,\bbD^n$. The real plane $\bbR^2$ is parametrized by coordinates $(x,y)$, while the projective real plane $\bbP^2(\bbR)$ has homogeneous coordinates $(x:y:z)$.

    The set $\{1,2,\dots,n\}$ is denoted $[n]$.

    A tuple/list of elements of a set $A$ is an element $a\in A^{<\infty}=\bigcup_{i=0}^\infty A^i$. The only difference between a tuple and a list is that to denote its components, the notation $a_i$ is used for the former and $a[i]$ for the latter. The concatenation of two tuples/lists $a,a'\in A^{<\infty}$ is denoted $a^\frown a'$.

    A Dyck word is a list of elements of $\{\text{``(''},\text{``)''}\}$ in which the same number of ``('' and ``)'' is used, and at no point of the list there are more preceding ``)'' than ``(''.

    Borrowing dependent pairs from type theory, for a set $A$ and an indexed family of sets $\{B_a\}_{a\in A}$ we use the following notation:
    \[\sum_{a\in A}B_a:=\{(a,b)\mid a\in A,b\in B_a\}.\]

\subsection{Acknowledgments and AI usage}
I thank my advisor Sandra Di Rocco for the many meetings and discussions we had, and for her patience through the whole project. The generative AI models ChatGPT and Claude were used to implement some small code functions, research the literature, assist with the writing, proof-reading and for finding a realization to a particular combinatorial curve (§\ref{subsec:cubic}). ChatGPT was also helpful for coming up with the proofs of some of the topological lemmas.

\section{Cellular decompositions of a surface}\label{sec:cell}

Our treatment of surfaces and curves is purely topological - smoothness is never mentioned - although we require submanifolds to be ``neat'' in the sense of \cite{hirsch2012differential}[1, §4]. The given topological definition of transverse intersection reflects the definition of neat submanifold, and is weaker than the standard $C^1$ definition involving tangent spaces, but it is sufficient for our needs.

\begin{definition}
    A \textbf{surface} $\calS$ is a compact 2-dimensional topological manifold, with boundary $\partial\calS$ and interior $\interior\calS$. A \textbf{curve on $\calS$} is a closed set $\calC\subset\calS$ which is locally flat, i.e.\ for all $p\in\calC$ there exists an open neighborhood $\calU\ni p$ and a homeomorphism
    \[(\calU,\calC\cap\calU)\simeq
        \begin{cases}
            (\bbR^2,\{y=0\})&\quad\text{if $p\in\interior\calS$}\\
            (\{x\ge0\},\{y=0\})&\quad\text{if $p\in\partial\calS$}.
        \end{cases}\]
    A curve is not required to be connected, i.e.\ it can be empty or can have multiple components, which are necessarily homeomorphic to $\bbS^1$ if disjoint with $\partial\calS$ (in which case they are called \textbf{ovals}), or to $\bbD^1$ with the two distinct boundary points in $\partial\calS$.

    Two curves $\calC,\calD$ on $\calS$ have \textbf{transverse intersection} (in symbols $\calC\pitchfork\calD$) if
    \begin{itemize}
        \item the intersection $\calC\cap\calD$ is finite;
        \item for every $p\in\calC\cap\calD$, there exists an open neighborhood $\calU\ni p$ and a homeomorphism $(\calU,\calC\cap\calU,\calD\cap\calU)\simeq(\bbR^2,\{x=0\},\{y=0\})$.
    \end{itemize}
    A collection of curves $\calC_1,\dots,\calC_k$ on $\calS$ have \textbf{transverse intersection} if they have pairwise transverse intersection and no three of them have a common point.

    A \textbf{$k$-arrangement} on $\calS$ is a $k$-tuple $(\calC_1,\dots,\calC_k)$ of curves on $\calS$ that intersect transversely. The \textbf{regions} determined by the arrangement are the connected components of $\interior\calS\setminus\bigcup_i\calC_i$.
\end{definition}

The following definition is a slight strengthening of the usual definition of CW complex \cite{hatcher2002algebraic}[Appendix] of dimension 2.

\begin{definition}
    A \textbf{cellular decomposition} of $\calS$ is a tuple $(\calV,\calE,\calF)$ coming from a filtration
    \[\varnothing=\calX_{-1}\subset\calX_0\subset\calX_1\subset\calX_2=\calS\]
    of closed sets, such that
    \begin{itemize}
        \item $\calV=\calX_0$ is the set of \textbf{vertices},
        \item $\calE$ is the set of connected components of $\calX_1\setminus\calX_0$, which are called \textbf{edges},
        \item $\calF$ is the set of connected components of $\calX_2\setminus\calX_1$, which are called \textbf{faces},
    \end{itemize}
    together with fixed continuous maps $\varphi_e\colon\bbD^1\to\overline{e}$ and $\varphi_f\colon\bbD^2\to\overline{f}$ called \textbf{characteristic maps} such that
    \begin{itemize}
        \item $\varphi_e|_{\bbB^1}\colon \bbB^1\to e$ is a homeomorphism, $\varphi_e(\bbD^1)=\overline{e}$ and $\varphi_e(\partial \bbD^1)\subseteq\calV$ for all edges $e\in\calE$;
        \item $\varphi_f|_{\bbB^2}\colon \bbB^2\to f$ is a homeomorphism, $\varphi_f(\bbD^2)=\overline{f}$ and $\varphi_f(\partial \bbD^2)$ is a union of elements of $\calE$ and $\calV$ for all faces $f\in\calF$.
    \end{itemize}
    We assume $\calV,\calE,\calF$ to be finite (which implies the ``weak topology'' part of the definition of CW complexes \cite{hatcher2002algebraic}[Page 521]) and that $\partial\calS$ is the union of elements of $\calV$ and $\calE$.
    
    Furthermore, given a characteristic map $\varphi_f$, consider the boundary $\partial f=\varphi_f(\partial \bbD^2)$ of a face $f\in\calF$. Let
    \begin{align*}
        \partial\varphi_{\bbD^2}\colon[0,1]&\to\partial\bbD^2\\
        t&\mapsto(\cos(2\pi t),\sin(2\pi t)).
    \end{align*}
    and
    \begin{align*}
        \partial\varphi_f:=\varphi_f\circ\partial\varphi_{\bbD^2}\colon[0,1]&\to\partial f\\
        t&\mapsto\varphi_f((\cos(2\pi t),\sin(2\pi t))).
    \end{align*}
    We do not require the CW complex to be regular; however, we further assume the following:
    \begin{itemize}
        \item $\partial\varphi_f(0)\in\calV$ is a vertex for all faces $f\in\calF$;
        \item $\varphi_f^{-1}(\calV)$ is a finite set of points;
        \item through $\varphi_f$, each component of $\partial\bbD^2\setminus\varphi_f^{-1}(\calV)$ maps homeomorphically onto an edge.
    \end{itemize}

    A curve $\calD$ on $\calS$ is \textbf{transverse to $(\calV,\calE,\calF)$} if $\calD\cap\calV=\varnothing$ and $\calD$ has transverse intersection with $\bigcup\calE$ inside $\interior\calS\setminus\calV$ (where $\bigcup\calE$ forms a curve). The \textbf{regions} determined by $\calD\pitchfork(\calV,\calE,\calF)$ are the connected components of $\interior\calS\setminus(\bigcup\calE\cup\calV\cup\calD)$.
\end{definition}

\begin{remark}
    As for general CW complexes, each characteristic map $\varphi_f$ is a closed quotient map, and $\varphi_f(\partial\bbD^2)=\partial f$.
\end{remark}

Given a $k$-arrangement $(\calC_1,\dots,\calC_k)$, let $\calC_0=\partial\calS$ and consider the tuple $(\calV,\calE,\calF)$ where
\begin{itemize}
    \item the vertices $\calV$ are the points of
    \[\bigcup_{0\le i<j\le k}\calC_i\cap\calC_j;\]
    \item the edges $\calE=\calE_0\cup\dots\cup\calE_k$ are labeled from 0 to $k$, where the elements of $\calE_i$ are the connected components of $\calC_i\setminus\calV$;
    \item the faces $\calF$ are the connected components of $\calS\setminus (\bigcup_{i=0}^k\calE_i\cup\calV)$.
\end{itemize}
If $e\in\calE_i$ we say that $e$ has \textbf{index} $i$.


\begin{definition}\label{def:cellular}
    An arrangement $(\calC_1,\dots,\calC_k)$ is \textbf{cellular} if the induced tuple $(\calV,\calE,\calF)$ is a cellular decomposition.
\end{definition}

\begin{example}\label{ex:lines_xyz1}
    On $\bbP^2(\bbR)$ with homogeneous coordinates $(x:y:z)$ consider the lines
    \[\calL_x=\{x=0\},\quad\calL_y=\{y=0\},\quad\calL_z=\{z=0\}.\]
    Then the arrangement of three lines $(\calL_x,\calL_y,\calL_z)$ and of two lines $(\calL_x,\calL_z)$ are cellular. They form the cellular decompositions resp.\ $(\calV^{(3)},\calE^{(3)},\calF^{(3)})$, $(\calV^{(2)},\calE^{(2)},\calF^{(2)})$ shown in Figure \ref{fig:lines_xyz-lines_xz-lines_z}, where the orientations of the edges and the faces are also fixed. The arrangement of one line $(\calL_z)$ is not cellular, since the only edge has no boundary.

    \begin{figure}
        \centering
        \includegraphics[width=1\linewidth]{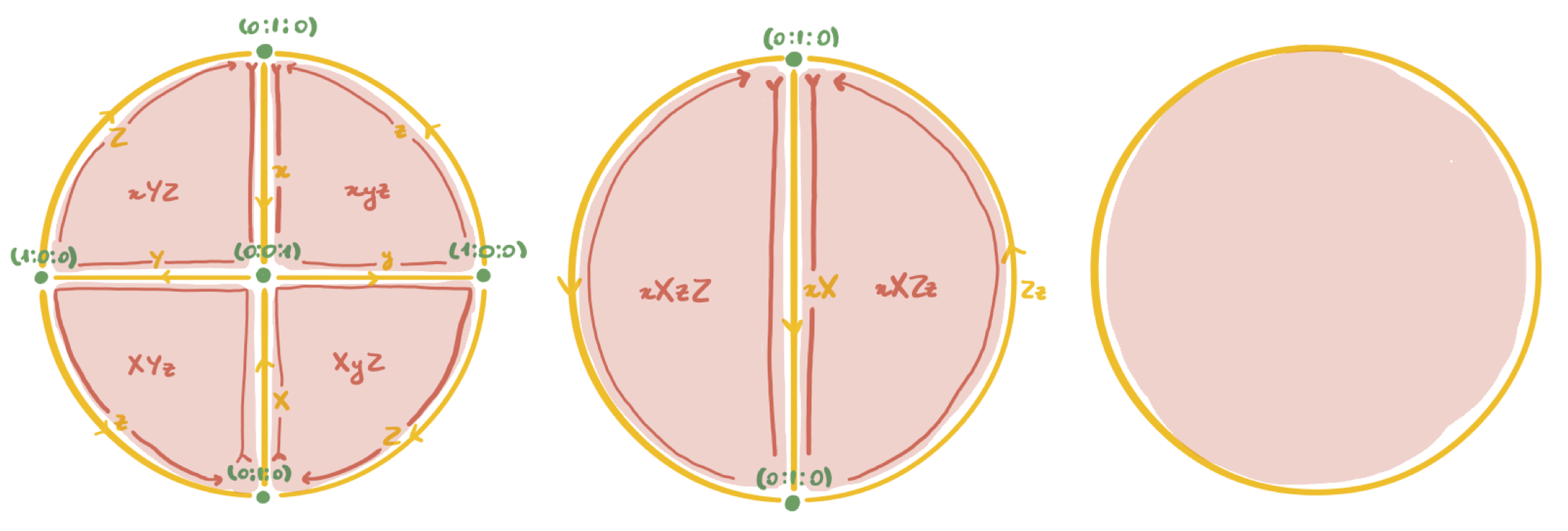}
        \caption{The cellular decompositions $(\calV^{(3)},\calE^{(3)},\calF^{(3)})$, $(\calV^{(2)},\calE^{(2)},\calF^{(2)})$ and the arrangement of one line.}
        \label{fig:lines_xyz-lines_xz-lines_z}
    \end{figure}
\end{example}

\begin{definition}
    A homeomorphism $\eta\colon\calS\to\calS$ is \textbf{compatible with $(\calV,\calE,\calF)$} if it restricts to permutations of these sets, and furthermore it \textbf{fixes $(\calV,\calE,\calF)$} if these permutations are the identity and for all edges $e\in\calE$ the characteristic map $\varphi_e$ has the same orientation as $\eta\circ\varphi_e$. Two $k$-arrangements $(\calC_1,\dots,\calC_k),(\calC'_1,\dots,\calC'_k)$ are \textbf{topologically equivalent} if there exists a homeomorphism $\eta\colon\calS\to\calS$ such that $\eta(\calC_i)=\calC'_i$ for all $i\in[k]$.
\end{definition}

\subsection{Adding a curve to a cellular decomposition}\label{subsec:adding}

Let $\calD\subset\calS$ be a curve transverse to a cellular decomposition $(\calV,\calE,\calF)$ on $\calS$. Let $\gr(\calD)$ be the union of the components of $\calD$ meeting $\bigcup\calE$, and let $\fl(\calD)=\calD\setminus\gr(\calD)$. Consider the tuple $(\calV_{\gr(\calD)},\calE_{\gr(\calD)},\calF_{\gr(\calD)})$ of $\calS$ where
\begin{itemize}
    \item $\calV_{\gr(\calD)}=\calV\cup\big(\gr(\calD)\cap\bigcup\calE\big)$;
    \item $\calE_{\gr(\calD)}$ are the connected components of $\bigcup\calE\cup\gr(\calD)\setminus\calV_{\gr(\calD)}$;
    \item $\calF_{\gr(\calD)}$ are the connected components of $\calS\setminus(\bigcup\calE_{\gr(\calD)}\cup\calV_{\gr(\calD)})$.
\end{itemize}

\begin{theorem}\label{thm:gr}
    $(\calV_{\gr(\calD)},\calE_{\gr(\calD)},\calF_{\gr(\calD)})$ forms a cellular decomposition of $\calS$ (up to a choice of characteristic maps).
\end{theorem}

We say that $\gr(\calD)$ is the \textbf{ground part of $\calD$} and that $\fl(\calD)$ is the \textbf{floating part of $\calD$}. Similarly, the connected components of $\calS\setminus(\calV\cup\calE\cup\calD)$ are divided into $\textbf{ground regions}$ if their closure meet $\bigcup\calE$, and \textbf{floating regions} otherwise. We call \textbf{ground faces} the elements of $\calF_{\gr(\calD)}$: each ground face contains exactly one ground region and some floating regions, whose arrangement will be encoded in the floating trees in Definition \ref{def:combcurve}. If $\fl(\calD)$ is empty, we say that $\calD$ is \textbf{floatless}.

    \begin{figure}
        \centering
        \includegraphics[width=0.25\linewidth]{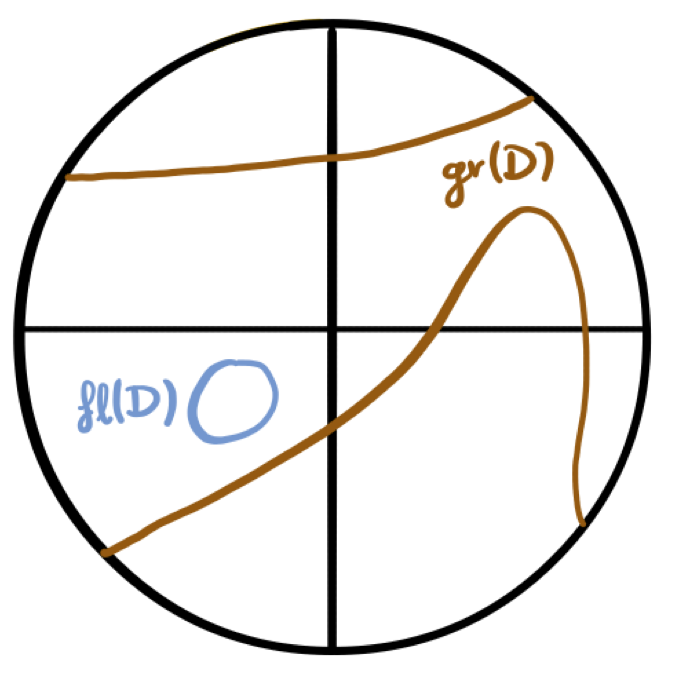}
        \caption{The ground and floating parts of a curve.}
        \label{fig:ground-floating-parts}
    \end{figure}


We first need some definitions and lemmas which will also be useful for the following.

\begin{definition}
Given a cellular decomposition $(\calV,\calE,\calF)$ of $\calS$, for any face $f\in\calF$ let $l(f)$ be the number of points in $\varphi_f^{-1}(\calV\cap\partial f)\subset\partial\bbD^2$. We temporarily set the numbers $0=t_0<t_1<\dots<t_{l(f)}=1$ such that
\[\varphi_f^{-1}(\calV\cap\partial f)\subset\partial\bbD^2=\{\partial\varphi_{\bbD^2}(t_0),\partial\varphi_{\bbD^2}(t_1),\dots,\partial\varphi_{\bbD^2}(t_{l(f)})\}.\]
Then, $\partial\varphi_f((t_{i-1},t_i))=e\in\calE$ is an edge, and since $\partial\varphi_f|_{(t_{i-1},t_i)}$ is a homeomorphism by our assumptions, the composition
\[\varphi_e^{-1}\circ\partial\varphi_{(t_{i-1},t_i)}\colon(t_{i-1},t_i)\to\bbB^1=(-1,1)\] is strictly monotone. So for any $i=1,\dots,l(f)$ we define
\[\{+,-\}\times\calE\ni \partial_i(f):=\begin{cases}
        +e&\text{if $\varphi_e^{-1}\circ\partial\varphi_{(t_{i-1},t_i)}$ is increasing,}\\
        -e&\text{if $\varphi_e^{-1}\circ\partial\varphi_{(t_{i-1},t_i)}$ is decreasing.}
\end{cases}\]
The value $\partial_i(f)$ represents what is the $i$-th edge crossed while traveling along $\partial\varphi_f$ and whether this happens with the same orientation as $\varphi_e$ or the opposite. We also define the open triangle
\[\Delta_i(f):=\varphi_f\big(\{(r\cos(2\pi t),r\sin(2\pi t))\mid r\in(0,1),t\in(t_{i-1},t_i)\}\big)\]
to be the image under $\varphi_f$ of the $i$-th ``slice'' of the disk.
\end{definition}

An important consequence of our definitions is that $\overline{\Delta_i(f)}\supset \partial_i(f)$.

\begin{figure}
    \centering
    \includegraphics[width=0.7\linewidth]{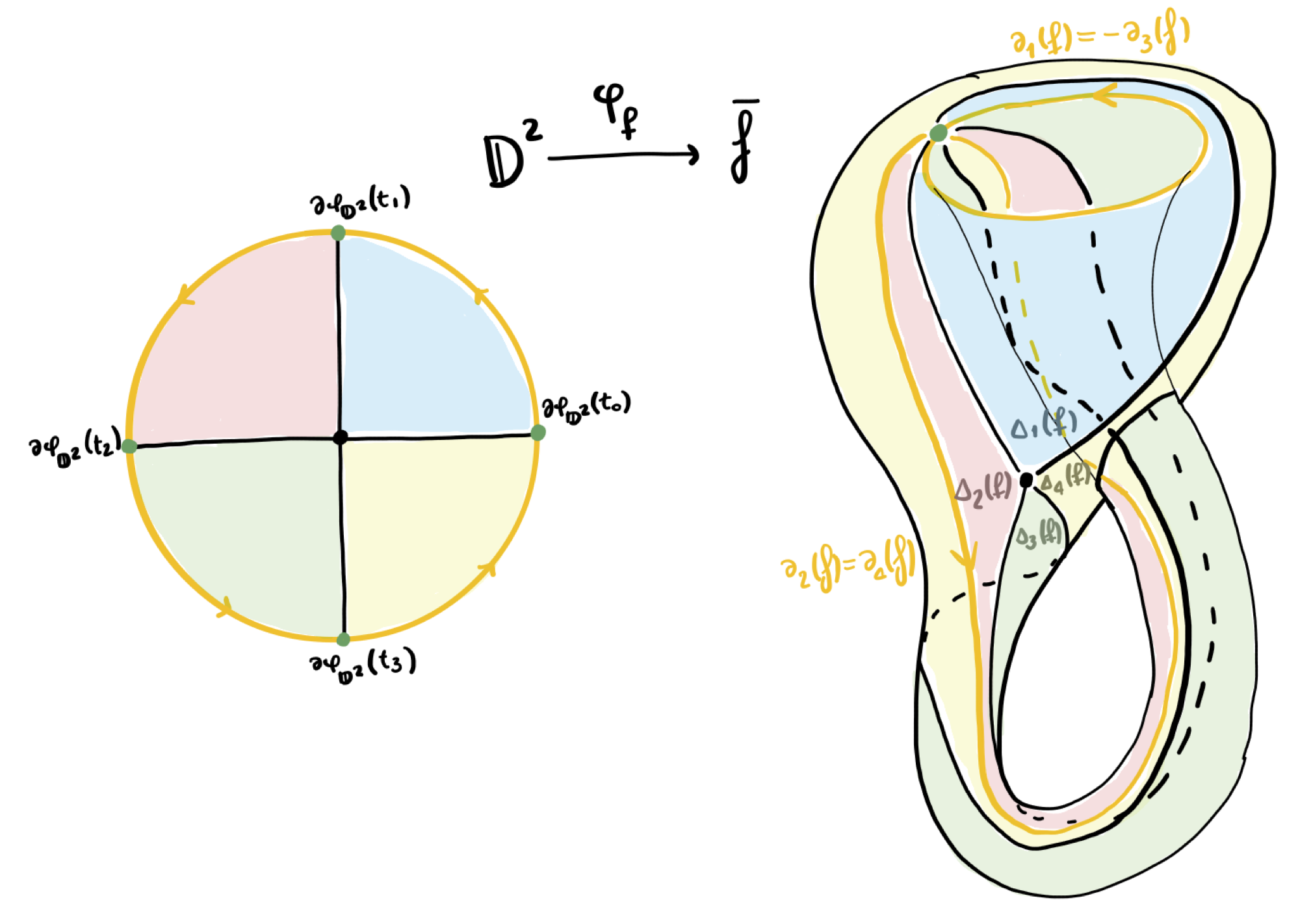}
    \caption{The values of $\partial_i(f),\Delta_i(f)$ for a face whose closure is homeomorphic to the Klein bottle.}
    \label{fig:klein}
\end{figure}

\begin{lemma}\label{lem:jordan}
    Let $\calC\subset\bbD^2$ be a connected curve such that $\calC\cap\partial\bbD^2\neq\varnothing$ (so $\calC\simeq\bbD^1$ and $\calC\cap\partial\bbD^2=\{p,q\}$ is made up of two points). Then $\calC$ splits $\bbB^2$ into two connected components $\calU_1,\calU_2$, each homeomorphic to $\bbB^2$ and whose closure is homeomorphic to $\bbD^2$.
    
    Moreover, up to homeomorphism of $\bbD^2$ fixing $\partial\bbD^2$ pointwise, the curve $\calC$ is determined by $\calC\cap\partial\bbD^2$.
\end{lemma}

\begin{proof}
    Let $\calL_1,\calL_2$ be the two connected components of $\partial\bbD^2\setminus\{p,q\}$. Then $\calC\cup\calL_1$ and $\calC\cup\calL_2$ are two closed curves which by the Jordan-Schönflies curve theorem \cite{thomassen1992jordan} enclose two subsets resp.\ $\calU_1,\calU_2$ both homeomorphic to $\bbB^2$ and whose closure is homeomorphic to $\bbD^2$; this proves the first part of the assertion.

    For uniqueness, consider another curve $\calC'$ with the same properties as $\calC$ and such that $\calC'\cap\partial\bbD^2=\{p,q\}$. In the same way as before, the closed curves $\calC'\cup\calL_1,\calC'\cup\calL_2$ enclose resp.\ $\calU'_1,\calU'_2$. Since $\calC,\calC'$ by definition of curves on $\bbD^2$ are images of maps $\bbD^1\to\bbD^2$, it is immediate to find a homeomorphism $\calC\cup\calL_1\to\calC'\cup\calL_1$ which fixes $\calL_1$ pointwise, which can be extended to a homeomorphism $\eta_1\colon\overline{\calU_1}\to\overline{\calU'_1}$ fixing $\calL_1$ (as any homeomorphism $\partial\bbD^2\to\partial\bbD^2$ can be radially extended to a homeomorphism $\bbD^2\to\bbD^2$). In the same way we can obtain a homeomorphism $\eta_2\colon\overline{\calU_2}\to\overline{\calU'_2}$ fixing $\calL_2$ and which agrees with $\eta_1$ on $\calC$, so by gluing $\eta_1$ and $\eta_2$ we finally obtain a homeomorphism $\eta\colon\bbD^2\to\bbD^2$ fixing $\partial\bbD^2$ pointwise and such that $\eta(\calC)=\calC'$.

\end{proof}

\begin{lemma}\label{lem:curveD2}
    If $\calC$ is transversal to a cellular decomposition $(\calV,\calE,\calF)$ then for each face $f\in\calF$, $\varphi^{-1}_f(\calC)$ is a curve on $\bbD^2$.
\end{lemma}
\begin{proof}
    To verify that $\varphi_f^{-1}(\calC)$ is a curve, we need to distinguish between points on $\interior(\bbD^2)=\bbB^2$ and on $\partial\bbD^2$. If $p\in\varphi_f^{-1}(\calC)\cap\bbB^2$ then there exists an open neighborhood $p\in\calU\subset\calS$ such that $(\calU,\calU\cap\calC)\simeq(\bbR^2,\{y=0\})$, and up to restriction we can suppose $\calU\subseteq f$.
    Then, by setting $\calU'=\varphi^{-1}_f(\calU)$ we have that $\calU'$ is an open neighborhood of $p$ and
    \[(\calU',\calU'\cap\varphi_f^{-1}(\calC))\simeq(\calU,\calU\cap\calC)\simeq(\bbR^2,\{y=0\})\]
    as desired.

    If instead $p\in\varphi_f^{-1}(\calC)\cap\partial\bbD^2$, then $\varphi_f(p)\in\partial f\setminus\calV$ and in particular $\varphi_f(p)\in e$ for some edge $e\in\calE$. Let $i$ be such that
    \[p\in \calU:=\{(r\cos(2\pi t),r\sin(2\pi t)\mid r\in(0,1],t\in(t_{i-1},t_i)\},\]
    and notice that $\partial_i(f)=\pm e$, that $\varphi_f(\calU)=\Delta_i(f)\cup e$ and that $\calU$ is an open neighborhood of $p$ in $\bbD^2$. There are now two possibilities.
    \begin{enumerate}
        \item There exists another $i'$ such that $\partial_{i'}(f)=\pm e$. Then, $\Delta_i(f)\cup\Delta_{i'}(f)\cup e$ is an open neighborhood of $p$, inside which (by the fact that $\calC\pitchfork e$) we can find an open neighborhood $\calW$ such that
        \[(\calW,\calW\cap e,\calW\cap\calC)\simeq(\bbR^2,\{y=0\},\{x=0\}).\]
        Now, through this homeomorphism, $\calW\setminus e\simeq\{y\neq0\}$ and
        \[\calW\setminus e=(\calW\cap\Delta_i(f))\sqcup(\calW\cap\Delta_{i'}(f))\]
        so without loss of generality, $\calW\cap\Delta_i(f)\simeq\{y\ge0\}$. Let $\calW'=\calW\cap(\Delta_i(f)\cup e)$. Then,
        \begin{equation}\label{eq:messy1}
            (\calW',\calW'\cap e,\calW'\cap\calC)\simeq(\{y\ge0\},\{y=0\},\{x=0\}).
        \end{equation}
        Finally, $\varphi_f$ is injective when restricted to $\calU$. Additionally,
        \[\varphi_f^{-1}(\calW)\cap\calU=\varphi_f^{-1}(\calW')\cap\calU\]
        is an open neighborhood of $p$, and by taking $\calW$ small enough we can suppose that $\overline{\varphi_f^{-1}(\calW)\cap\calU}\subset\calU$. Then the restriction
        \[\psi:=\varphi_f|_{\overline{\varphi_f^{-1}(\calW)\cap\calU}}\]
        is closed map from a compact space to a Hausdorff space, so a homeomorphism with its image, which lets us conclude that
        \begin{equation}\label{eq:messy2}
            (\psi^{-1}(\calW),\psi^{-1}(\calW)\cap\partial\bbD^2,\psi^{-1}(\calW)\cap\varphi^{-1}(\calC))\simeq(\calW',\calW'\cap e,\calW'\cap\calC).
        \end{equation}
        By composing the homeomorphisms \eqref{eq:messy1} and \eqref{eq:messy2} we get what we wanted.

        \item There is no other $i'$ such that $\partial_{i'}(f)=\pm e$. We proceed similarly to before, the only difference being that this time we take the open neighborhood $\calW\subset\overline{f}$ so that
        \[(\calW,\calW\cap e,\calW\cap\calC)\simeq(\{y\ge0\},\{y=0\},\{x=0\}).\]
    \end{enumerate}
\end{proof}

\begin{proof}[Proof of Theorem \ref{thm:gr}]
    We need to show that for all faces $f\in\calF$ and for all connected components $g$ of $f\setminus\gr(\calD)$, there exists a characteristic map $\varphi_g\colon\bbD^2\to\overline{g}$, where in particular $\varphi_g(\partial\bbD^2)$ is a union of elements of $\calV_{\gr(\calD)}$ and $\calE_{\gr}$. Call arcs the connected components of $\varphi_f^{-1}(\gr(\calD))$. We proceed by induction on the number $n$ of arcs. If $n=0$, then $f=g$ and the proposition is obvious.

    For $n>0$, consider any arc $\alpha$. By Lemma \ref{lem:jordan}, it splits $\bbB^2$ into two open subsets $\calU_1,\calU_2$ such that $(\overline{\calU_i},\calU_i)\simeq(\bbD^2,\bbB^2)$. We show that $\psi:=\varphi_f|_{\overline{\calU_1}}$ is, after the identification of its source with $\bbD^2$, a characteristic map for the face $\varphi_f(\calU_1)$ of the decomposition $(\calV',\calE',\calF')$ of $\overline{f}$ where
    \begin{itemize}
        \item $\calV'=(\calV\cap\overline{f})\cup(\varphi_f(\overline\alpha)\cap\partial f)$;
        \item $\calE'$ is the set of connected components of $(\partial f\cup\varphi_f(\alpha))\setminus\calV'$, or equivalently the element $\varphi_f(\alpha)$ together with set of connected components of $\partial f\setminus\calV'$;
        \item $\calF'$ is the set of connected components of $\overline{f}\setminus(\bigcup\calE'\cup\calV')=f\setminus\varphi_f(\alpha)$, that is,
        \[\calF'=\{\varphi_f(\calU_1),\varphi_f(\calU_2)\}.\]
    \end{itemize}
    Similarly, $\varphi_f|_{\overline{\calU_2}}$ is a characteristic map for $\varphi_f(\calU_2)$. Hence, $(\calV',\calE',\calF')$ is a cellular decomposition of $\overline{f}$. Now, there are yet $<n$ arcs to consider on separately $\varphi_f|_{\overline{\calU_1}}$ and $\varphi_f|_{\overline{\calU_2}}$, but by inductive hypothesis after adding those arcs we get a cellular decomposition.

    We now prove that $\psi$ has the required properties.
    \begin{enumerate}
        \item $\psi$ is a homeomorphism restricted to $\calU_1$: obvious from being a restriction of $\varphi_f$, and $\calU_1\subset f$.
        \item $\psi$ is surjective: since $\varphi_f$ is a closed map, we conclude $\psi(\overline{\calU_1})=\overline{\psi(\calU_1)}$.
        \item $\psi(\partial\calU_1)$ is a union of elements of $\calV'$ and $\calE'$: using notation from the proof of Lemma \ref{lem:jordan}, $\partial\calU_1=\calL_1\cup\{p,q\}\cup\alpha$. We have that $\psi(\alpha)\in\calE'$, that $\psi(p),\psi(q)\in\calV'$, and $\psi(\calL_1)$ is a union of connected components of $\partial f\setminus\calV'$ together with elements of $\calV'$, as desired.
    \end{enumerate}
    Additionally, $\psi^{-1}(\calV')$ is a finite set (in fact, a subset of $\varphi^{-1}(\calV)$ together with $\{p,q\}$), and $\psi$ sends components of $\partial\calU_1\setminus\psi^{-1}(\calV')$ homeomorphically onto edges: for the elements of $\calL_1\setminus\psi^{-1}(\calV')$ this follows from $\varphi_f$ being a characteristic map, and $\psi|_{\alpha}$ is a homeomorphism since $\alpha\subset\bbB^2$.
\end{proof}



\section{From geometry to combinatorics}\label{sec:geocom}

This section develops the central idea behind this work, namely the translation in combinatorial terms of the geometric problem of describing topological types of curves transverse to cellular decompositions. The basic idea is not new - the resulting combinatorial objects are usually called \textit{dessin d'enfant}. There are also some similarities with the constructions presented in \cite{degtyarev2012topology}. The main contributions of this work are the application of this line of reasoning to arrangements of curves, the description of multiple criteria for the realizability of such dessin, and the implementation of each of these constructions in \texttt{Julia}.

\begin{definition}\label{def:ccc}
    Consider a tuple $(V,E,F)$ where $V,E,F$ are finite sets, together with functions $\partial_0,\partial_1\colon E\to V$ and a function $\partial_*\colon F\to(\{+,-\}\times E)^{<\infty}$.

    We say that $(V,E,F)$ is a \textbf{combinatorial cell complex on $\calS$} if it can be obtained from a cellular decomposition $(\calV,\calE,\calF)$ of $\calS$ in the following way:
    \begin{itemize}
        \item as sets, $V=\calV$, $E=\calE$ and $F=\calF$;
        \item for all $e\in\calE$,
        \[\partial_0(e)=\varphi_e(-1),\quad\partial_1(e)=\varphi_e(1);\]
        \item for all $f\in\calF$,
        \[\partial_*(f)=\big(\partial_1(f),\dots,\partial_{l(f)}(f)\big).\]
    \end{itemize}

    A homeomorphism $\eta\colon\calS\to\calS$ compatible with $(\calV,\calE,\calF)$ induces functions
    \begin{itemize}
        \item $\eta_V\colon V\to V$ where $\eta_V(v)=\eta(v)$;
        \item $\eta_E\colon E\to\{+,-\}\times E$ where $\eta_E(e)=+\eta(e)$ if $\varphi_{\eta(e)}$ and $\eta\circ\varphi_e$ have the same orientation, and $\eta_E(e)=-\eta(e)$ otherwise;
        \item $\eta_F\colon F\to\{+,-\}\times F$ where $\eta_F(f)=+\eta(f)$ if $\partial\varphi_{\eta(f)}$ and $\eta\circ\partial\varphi_f$ have the same orientation, and $\eta_F(f)=-\eta(f)$ otherwise.
    \end{itemize}
    The set of functions $(\eta_V,\eta_E,\eta_F)$ obtainable in this way are the \textbf{automorphisms} of $(V,E,F)$. The homeomorphism $\eta$ fixes $(\calV,\calE,\calF)$ precisely when $\eta_V=\id_V$, $\eta_E=+\id_E$ and $\eta_F=+\id_F$.

    If $(\calV,\calE,\calF)$ is the cellular decomposition constructed from a cellular arrangement $(\calC_1,\dots,\calC_k)$, we additionally consider the \textbf{index function} $i\colon E\to\bbN$ defined as $i(e)=i$ for all $e\in\calE_i$ (recall that $\calE_i\subseteq\calE$ are the connected components of $\calC_i\setminus\calV$).

    The \textbf{$i$-automorphisms} of $(V,E,F)$ are the automorphisms of $(V,E,F)$ fixing the function $i$, that is, $i\circ\eta_E=i$. These describe the effect that a homeomorphism inducing a topological equivalence of $(\calC_1,\dots,\calC_k)$ with itself has on the combinatorial structure of its decomposition.
\end{definition}

\begin{example}\label{ex:lines_xyz2}
    The cellular decomposition $(\calV^{(3)},\calE^{(3)},\calF^{(3)})$ coming from the arrangement $(\calL_x,\calL_y,\calL_z)$ induces the combinatorial cell complex $(V^{(3)},E^{(3)},F^{(3)})$ with functions $\partial$ and $i$ defined as
    \begin{gather*}
        \partial_0(x)=\partial_0(X)=(0:1:0),\quad\partial_1(x)=\partial_1(X)=(0:0:1),\quad i(x)=i(X)=1,\\
        \partial_0(y)=\partial_0(Y)=(0:0:1),\quad\partial_1(y)=\partial_1(Y)=(1:0:0),\quad i(y)=i(Y)=2,\\
        \partial_0(z)=\partial_0(Z)=(1:0:0),\quad\partial_1(z)=\partial_1(Z)=(0:1:0);\quad i(z)=i(Z)=3,\\
        \partial_*(xyz)=(+x,+y,+z),\\
        \partial_*(xYZ)=(+x,+Y,+Z),\\
        \partial_*(XyZ)=(+X,+y,+Z),\\
        \partial_*(XYz)=(+X,+Y,+z).
    \end{gather*}
    The group of automorphisms of $(V^{(3)},E^{(3)},F^{(3)})$ has 24 elements and can be identified with the group of the rotational symmetries of the octahedron. We can name each element $\eta$ of the group after the preimage of the edges $x,y,z\in E^{(3)}$ through $\eta_E$, obtaining the classification in Figure \ref{fig:automorphisms}. Each of these automorphisms comes from a projective transformation of the plane of the form
    \[(x:y:z)\mapsto(\pm T_1:\pm T_2 :T_3)\]
    where $\{T_1,T_2,T_3\}$ is a permutation of $\{x,y,z\}$. The subgroup of $i$-automorphisms has 4 elements.

    The combinatorial cell complex $(V^{(2)},E^{(2)},F^{(2)})$ coming from the cellular decomposition induced by the arrangement of two lines $(\calL_x,\calL_z)$ has functions $\partial,i$ defined as
    \begin{gather*}
        \partial_0(xX)=\partial_1(xX)=\partial_0(Zz)=\partial_1(Zz)=(0:1:0),\quad i(xX)=1,\;i(Zz)=3;\\
        \partial_*(xXZz)=(+xX,+Zz),\quad\partial_*(xXzZ)=(+xX,-Zz).
    \end{gather*}
    The group of automorphisms is of order 8, with the subgroup of $i$-automorphisms being of order 4.
\end{example}

\begin{figure}
    \centering
    \includegraphics[width=1\linewidth]{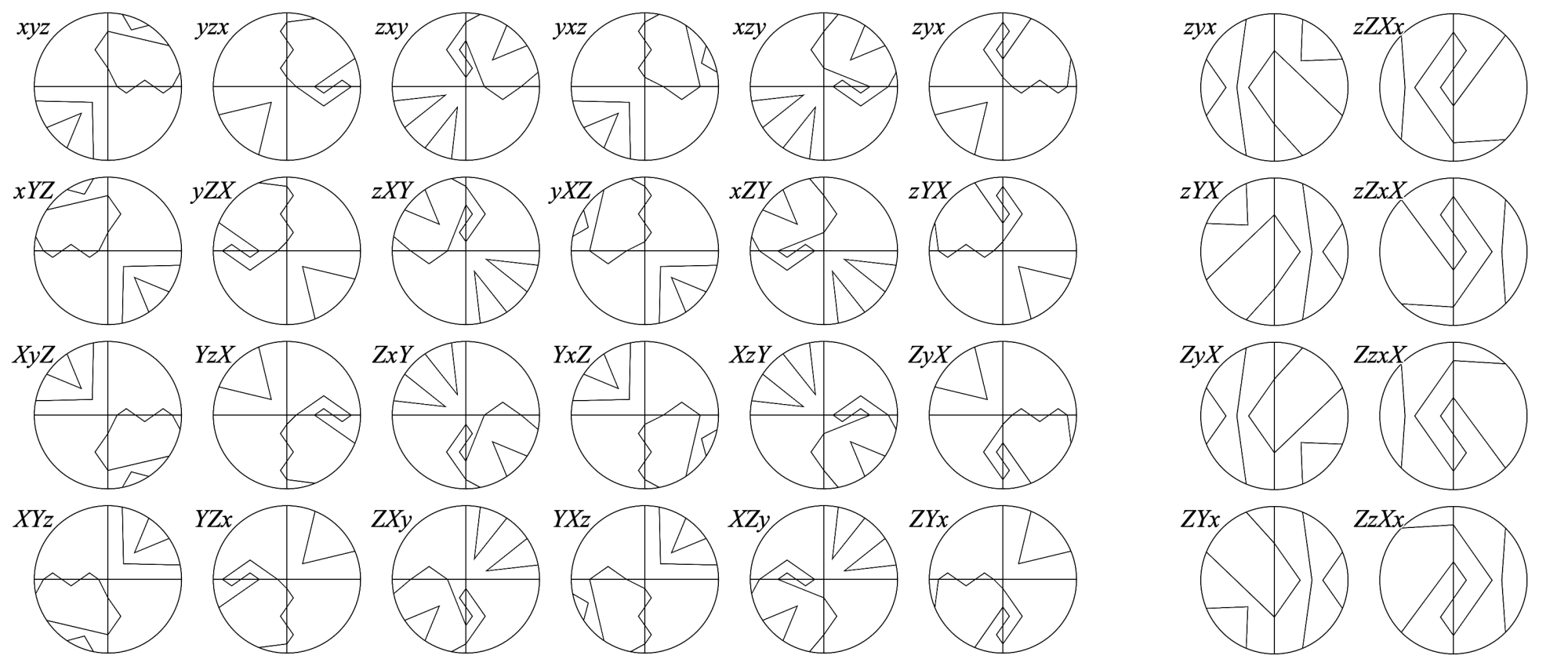}
    \caption{The group of automorphisms of $(V^{(3)},E^{(3)},F^{(3)})$ and $(V^{(2)},E^{(2)},F^{(2)})$, illustrated through their action on a combinatorial curve. In both cases the first column is the subgroup of $i$-automorphisms.}
    \label{fig:automorphisms}
\end{figure}

\begin{definition}
    Let $\calC$ be a curve on $\bbD^2$. The \textbf{Dyck word $W(\calC)$ associated to $\calC$} is the Dyck word constructed by starting with the empty word and while traveling along $\partial\varphi_{\bbD^2}$, adding a ``('' or a ``)'' to the end of the word whenever a component of $\calC$ is met resp.\ for the first or second time.
\end{definition}

\begin{definition}\label{def:combcurve}
    Let $(V,E,F)$ be a combinatorial cell complex on $\calS$. A \textbf{combinatorial curve} transverse to $(V,E,F)$ is given by the following data $(n,W,T)$:
    \begin{itemize}
        \item for every edge $e\in E$, a non-negative integer $n(e)$ such that $n(f):=\sum_{i=1}^{l(f)}n(e_i)$ is even for all faces $f\in F$;
        \item for every face $f\in F$, a Dyck word $W_f$ of length $n(f)$;
        \item for every $g\in\sum_{f\in F}[n(f)/2+1]$, a rooted tree $T_g$, called the \textbf{floating tree} on $g$.
    \end{itemize}

    From a curve $\calD$ transverse to $(\calV,\calE,\calF)$ we construct the following combinatorial curve $(n,W,T)$ transverse to $(V,E,F)$:
    \begin{itemize}
        \item for every edge $e\in\calE$, let $n(e):=\#(\calD\cap e)$ be the number of points of intersection of $e$ with the curve $\calD$;
        \item for every face $f\in\calF$, let $W_f=W(\varphi_f^{-1}(\calD))$ be the Dyck word associated to the curve $\varphi_f^{-1}(\calD)$ on $\bbD^2$;
        \item for every face $f\in\calF$, consider the ground faces $g\in\calF_{\gr(\calD)},g\subseteq f$ in the order in which each $\overline{\varphi_f^{-1}(g)}$ is first met by traveling along $\partial\varphi_{\bbD^2}$. If $g$ is the $i$-th, let $T_{(f,i)}$ be the rooted tree formed by the regions of $\calD$ contained inside $g$, where the root is the unique ground region whose closure contains $\partial g$.
    \end{itemize}  
\end{definition}

For an example see Figure \ref{fig:lines_xyz}.

\begin{remark} \label{rem:conventions}
    In examples and in the library \jlinl{NWT}, an ordering of the set $F$ of faces is necessarily chosen. This lets us index the rooted trees $T_{(f,i)}$ using a single index $g$, which represents the lexicographic ordering of the pairs $(f,i)$. Furthermore, in the library, each rooted tree $T_g$ is encoded as a pair $(g,W)$ where $W$ is the Dyck word encoded by a depth-first traversal of the rooted tree, where backtracking is recorded. For such a traversal, an ordering of siblings in each vertex of the tree is required: we use a canonical order, given by recursively sorting children. In this way, two combinatorial curves $(n,W,T),(n',W',T')$ transversal to the same combinatorial cell complex are ``mathematically'' equal if and only if, as \texttt{Julia} objects, $n=n'$, $W=W'$ and $T=T'$.
\end{remark}

\begin{remark}\label{rem:act}
    Let $\calD$ be a curve intersecting $(\calV,\calE,\calF)$ transversely, and consider the induced combinatorial curve $(n,W,T)$. Given a homeomorphism $\eta$, the curve $\eta(\calD)$ will have an induced combinatorial curve $(n',W',T')$ which depends uniquely on $(n,W,T)$ and the functions $\eta_V$, $\eta_E$ and $\eta_F$. In this way, the group of automorphisms of $(V,E,F)$ acts on the combinatorial curves transverse to $(V,E,F)$. Clearly, a homeomorphism $\eta$ fixes $(\calV,\calE,\calF)$ if and only if it induces the trivial action.
\end{remark}

\begin{lemma}\label{lem:disk}
    Up to homeomorphism fixing $\partial\bbD^2$ pointwise, a curve $\calC$ on $\bbD^2$ whose every component meets the boundary $\partial\bbD^2$ is determined by the set $\calC\cap\partial\bbD^2$ and the Dyck word $W(\calC)$.
\end{lemma}

\begin{proof}
    A Dyck word is the same as an ordered tree, i.e.\ a rooted tree in which an order is specified for every set of children of a vertex. In our case, each non-root vertex of this tree corresponds to a component of $\calC$, which we can see as an arc $\alpha\colon[0,1]\to\bbD^2$ where the endpoints $\alpha(0),\alpha(1)$ are determined by $W(\calC)$ and the set $\calC\cap\partial\bbD^2$, and where $\alpha(t)\in\bbB^2$ for all $0<t<1$. Then, by traversing the tree depth-first, we show that up to homeomorphism there is only one way to place each arc $\alpha$ at each step. Indeed, by Lemma \ref{lem:jordan}, this arc is unique up to homeomorphism and splits the region containing it in two regions over which the same lemma applies.
\end{proof}

\begin{lemma}\label{lem:float}
    Up to homeomorphism fixing $\partial\bbD^2$ pointwise, a curve $\calC$ on $\bbD^2$ disjoint from $\partial\bbD^2$ is determined by the rooted tree formed by the connected components of $\bbB^2\setminus\calC$, where the root is the unique component whose closure contains $\partial\bbD^2$.
\end{lemma}
\begin{proof}
    As in the proof of \ref{lem:disk}, we traverse the tree depth-first and show that at each step there is only one way up to homeomorphism to place each oval. By choosing such a traversal, the oval is always contractible inside the open region it is to be placed, so it is easy to see that the placement is unique up to homeomorphism of that region fixing its boundary.
\end{proof}

\begin{theorem}\label{thm:main}
    The construction in Definition \ref{def:combcurve} defines a bijection
    \begin{equation}\label{eq:main}
        \frac{\{\calD\mid\calD\pitchfork(\calV,\calE,\calF)\}}{\text{homeomorphism that fixes $(\calV,\calE,\calF)$}}\to\{(n,W,T)\mid(n,W,T)\pitchfork(V,E,F)\}
    \end{equation}
    between curves $\calD$ transverse to $(\calV,\calE,\calF)$ up to homeomorphism that fixes $(\calV,\calE,\calF)$, and combinatorial curves transverse to $(V,E,F)$.
\end{theorem}

\begin{proof}
    The function is well-defined: suppose $\calD$ and $\calD'$ are equivalent through a homeomorphism $\eta$ that fixes $(\calV,\calE,\calF)$. Then, by Remark \ref{rem:act}, they have the same associated combinatorial curve.

    To prove surjectivity, choose $n_e$ points on each edge $e$. For each face $f$, draw the unique non-crossing arc system prescribed by $W_f$. Since the endpoint labels agree on both sides of every edge, these arc systems glue to a curve $\gr(\calD)$. Finally, in each ground face $g$ insert floating ovals according to the rooted tree $T_g$.

    We now focus on injectivity. Suppose $\calD,\calD'$ are two curves meeting $(\calV,\calE,\calF)$ transversely to which we associate the same data $(n,W,T)$. Then up to a homeomorphism that fixes $(\calV,\calE,\calF)$ we can assume $\calD\cap e=\calD'\cap e$ for all $e\in\calE$. To do this, fix an edge $e$. Our goal is to find an open set $\calU\subset\calS$ such that
    \begin{itemize}
        \item $\calU\cap e$ is a connected set containing both $\calD\cap e,\calD'\cap e$ and such that the following equalities hold:
        \begin{equation}\label{eq:closureU}
            \overline{\calU}\cap\calV=\varnothing,\quad\overline{\calU}\cap\bigcup\calE=\overline{\calU}\cap e;
        \end{equation}
        \item it is ``easy'' to find a homeomorphism $\overline{\calU}\simeq\overline{\calU}$ fixing the boundary $\partial\calU$ pointwise, $\calU\cap e$ setwise and that sends $\calD\cap e$ into $\calD'\cap e$.
    \end{itemize}
    This homeomorphism fixes $(\calV,\calE,\calF)$ and achieves what we want. We construct $\calU$ explicitly, distinguishing two cases:
    \begin{itemize}
        \item[$\underline{e\not\subset\partial\calS}$:] by definition of 2-manifold, there exist exactly two tuples $(f,i),(f',i')$ such that $\partial_i(f),\partial_{i'}(f')=\pm e$. Recall that
        \[\partial\varphi_f|_{(t_{i-1},t_i)}\colon(t_{i-1},t_i)\to e\]
        is a homeomorphism. Pick values $a,b\in(t_{i-1},t_i)$ such that \[(\partial\varphi_f|_{(t_{i-1},t_i)})^{-1}((\calD\cap e)\cup(\calD'\cap e))\subset(a,b).\]
        Define
        \[T=\varphi_f(\{(r\cos(2\pi t),r\sin(2\pi t))\mid r\in(0,1),t\in(a,b)\})\subset\Delta_i(f).\]
        Next, let
        \[a'=(\partial\varphi_{f'}|_{(t_{i'-1},t_{i'})})^{-1}(\partial\varphi_f|_{(t_{i-1},t_i)}(a)),\quad b'=(\partial\varphi_{f'}|_{(t_{i'-1},t_{i'})})^{-1}(\partial\varphi_f|_{(t_{i-1},t_i)}(b))\]
        and likewise define
        \[T'=\varphi_{f'}(\{(r\cos(2\pi t),r\sin(2\pi t))\mid r\in(0,1),t\in(a',b')\})\subset\Delta_{i'}(f').\]
        Finally, let
        \[\calU=T\cup T'\cup\partial\varphi_f((a,b)).\]
        Then, we have that $\calU$ is open, $\calU\cap e=\partial\varphi_f((a,b))$ is connected containing both $\calD\cap e,\calD'\cap e$, which satisfies the equalities \eqref{eq:closureU}. Moreover,
        \[(\overline{\calU},\calU,\calU\cap e)\simeq(\bbD^2,\bbB^2,0\times(-1,1))\]
        so it is easy to find a homeomorphism $\overline{\calU}\to\overline{\calU}$ with the desired properties.
        \item[$\underline{e\subset\partial\calS}$:] there exists a unique tuple $(f,i)$ with $\partial_i(f)=\pm e$. Pick $a,b$ and define $T$ as before. Let
        \[\calU=T\cup\partial\varphi_f((a,b)).\]
        Now we instead have that
        \[(\overline{\calU},\calU,\calU\cap e)\simeq(\bbD^2\cap\{y\ge0\},\bbB^2\cap\{y\ge0\},0\times[0,1)),\]
        but everything works as before.
    \end{itemize}

    Next we show that for each $f\in\calF$, there is a homeomorphism $\eta_f\colon\overline{f}\to\overline{f}$ fixing $\partial f$ such that $\eta_f(\gr(\calD)\cap\overline{f})=\gr(\calD')\cap\overline{f}$. Indeed, the curves $\varphi_f^{-1}(\gr(\calD)\cap\overline{f}), \varphi_f^{-1}(\gr(\calD')\cap\overline{f})$ on $\bbD^2$ have the same Dyck word $W_f$ associated to them in Lemma \ref{lem:disk}, so there is a homeomorphism $\eta\colon\bbD^2\to\bbD^2$ fixing $\partial\bbD^2$ sending the first curve to the second. Since $\eta$ fixes $\partial\bbD^2$ pointwise, it preserves the fibers of $\varphi_f$. Therefore $\eta$ descends to a homeomorphism $\eta_f\colon\overline{f}\to\overline{f}$ satisfying $\eta_f\circ\varphi_f=\varphi_f\circ\eta$, which has the desired properties. Since each $\eta_f$ is the identity on $\partial f$, these glue together to form a homeomorphism $\eta\colon\calS\to\calS$ fixing $(\calV,\calE,\calF)$ and such that $\eta(\gr(\calD))=\gr(\calD')$.

    We can now suppose $\gr(\calD)=\gr(\calD')$. Similarly to before, using Lemma \ref{lem:float} we conclude that for each ground face $g\in\calF_{\gr(\calD)}$, the floating part $\fl(\calD)\cap g$ is uniquely specified by the rooted tree $T_g$ up to homeomorphism of $\overline{g}$ fixing its boundary.
\end{proof}

\begin{corollary}[Theorem A]\label{cor:main_eq}
    Suppose $(\calC_1,\dots,\calC_k)$ is a cellular $k$-arrangement. Then, the construction in Definition \ref{def:combcurve} defines a bijection
    \begin{equation}\label{eq:main_eq}
        \frac{\{(\calC_1,\dots,\calC_k,\calD)\mid\calD\pitchfork(\calC_1,\dots,\calC_k)\}}{\text{topological type}}\to\frac{\{(n,W,T)\mid(n,W,T)\pitchfork(V,E,F)\}}{\text{$i$-automorphism of $(V,E,F)$}}
    \end{equation}
    between topological types of $(k+1)$-arrangements $(\calC_1,\dots,\calC_k,\calD)$, and combinatorial curves transverse to $(V,E,F)$ up to automorphisms of $(V,E,F)$ fixing the index function $i\colon E\to\bbN$.
\end{corollary}

\begin{proof}
    By Theorem \ref{thm:main}, the correspondence \eqref{eq:main} is bijective. By taking quotients, it is clear that also the correspondence \eqref{eq:main_eq} is both well-defined and bijective.
\end{proof}

\begin{example}\label{ex:unlabeled}
    In the case $\calS=\bbP^2(\bbR)$, consider arrangements of three lines and a curve $(\calL_1,\calL_2,\calL_3,\calC)$. Up to homeomorphisms, there is only one arrangement of three (pseudo)lines $(\calL_1,\calL_2,\calL_3)$ which all meet each other exactly once (e.g.\ the one described in Example \ref{ex:lines_xyz1}), and this arrangement is cellular, so
    \begin{itemize}
        \item the set of combinatorial curves $(n,W,T)$ \textit{up to $i$-automorphism} is in bijection with the topological types of arrangements $(\calL_1,\calL_2,\calL_3,\calC)$;
        \item the set of combinatorial curves $(n,W,T)$ \textit{up to automorphism} (not necessarily fixing the index function) is in bijection with the topological types of arrangements of three \textbf{unlabeled} lines and a curve, where we also allow the three lines to be permuted by the ambient homeomorphism.
    \end{itemize}
\end{example}

\begin{remark}\label{rem:ambient}
    We say that two arrangements $(\calC_1,\dots,\calC_k)$, $(\calC'_1,\dots,\calC'_k)$ are \textbf{ambient isotopic} if there exists a map $\Phi\colon[0,1]\times\calS\to\calS$ such that $\Phi_t=\Phi(t,\cdot)\colon\calS\to\calS$ is a homeomorphism for all $t\in[0,1]$, $\Phi_0=\id_{\calS}$ and $\Phi_1(\calC_i)=\calC'_i$ for all $i\in[k]$. In general, being ambient isotopic is a stronger notion than being topologically equivalent.

    The appropriate modification of the bijection \eqref{eq:main_eq} to the ambient isotopy relation is
    \begin{equation}
        \frac{\{(\calC_1,\dots,\calC_k,\calD)\mid\calD\pitchfork(\calC_1,\dots,\calC_k)\}}{\text{ambient isotopy}}\to\frac{\{(n,W,T)\mid(n,W,T)\pitchfork(V,E,F)\}}{\substack{\text{$i$-automorphism of $(V,E,F)$}\\\text{isotopic to the identity}}}
    \end{equation}
    where the $i$-automorphisms of $(V,E,F)$ isotopic to the identity form a subgroup of the $i$-automorphisms of $(V,E,F)$.

    In the case of $\calS=\bbP^2(\bbR)$, which is our main concern in the following, homeomorphisms and ambient isotopies coincide since $\bbP^2(\bbR)$ has trivial mapping class group \cite{beguin2020fixed}[Proposition 8.8]. Thus, our classification of topological types of arrangements of curves on $\bbP^2(\bbR)$ is the same as the classification up to ambient isotopy.
\end{remark}

\section{Basic combinatorial operations}\label{sec:ops}

This section contains detailed algorithms regarding geometric operations which can be carried out purely combinatorially. While the algorithms themselves may be technical, the examples should make it clear what the purpose of each operation is.

\subsection{Refinement of a combinatorial cell complex}\label{subsec:refinement}

Let $(n,W,T)$ be a combinatorial curve transverse to a combinatorial cell complex $(V,E,F)$. Let $\calD,(\calV,\calE,\calF)$ be a curve and a cellular decomposition inducing them. in §\ref{subsec:adding} we constructed the cellular decomposition $(\calV_{\gr(\calD)},\calE_{\gr(\calD)},\calF_{\gr(\calD)})$: it turns out that the induced combinatorial cell complex $(V_{\gr},E_{\gr},F_{\gr})$ depends only on $(V,E,F)$ and $(n,W,T)$. We call $(V_{\gr},E_{\gr},F_{\gr})$ the \textbf{$(n,W,T)$-refinement of $(V,E,F)$}.
We now describe a procedure to calculate $(V_{\gr},E_{\gr},F_{\gr})$.

In the following, $V_{\gr},E_{\gr},F_{\gr}$ are interpreted as lists which get progressively filled with ``new'' elements, as opposed to the ``old'' elements in $V,E,F$. We start with Algorithm \ref{alg:addE} by determining $V_{\gr}$ and inserting into $E_{\gr}$ the edges of $\calE_{\gr(\calD)}$ contained in some edge of $\calE$. We also define a function $\gr\colon E\to E_{\gr}^{<\infty}$ associating to an old edge $e\in E$ the tuple of new edges contained in it, which by definition have the same orientation.

\begin{algorithm}
\caption{Filling $V_{\gr},E_{\gr}$ and the values of $\partial_0,\partial_1\colon E_{\gr}\to V_{\gr}$, $\gr\colon E\to E_{\gr}^{<\infty}$}\label{alg:addE}
\begin{algorithmic}
\State $V_{\gr} \gets V$
\State $E_{\gr} \gets []$
\State $F_{\gr} \gets []$
\For{$e\in E$}
\State insert new elements $v_1',\dots,v_{n(e)}'$ into $V_{\gr}$
\State insert new elements $e_0',\dots,e_{n(e)}'$ into $E_{\gr}$
\State $\partial_0(e_0') \gets \partial_0(e)$
\For{$i$ from 1 to $n(e)$}
\State $\partial_0(e'_i) \gets v'_i$
\State $\partial_1(e'_{i-1}) \gets v'_i$
\EndFor
\State $\partial_1(e'_{n(e)}) \gets \partial_1(e)$
\State $\gr(e) \gets (e_0',\dots,e_{n(e)}')$
\EndFor
\end{algorithmic}
\end{algorithm}

Fix a face $f\in F$; what follows is to be repeated for every face. If $n(f)=0$ then we just add a new element $f$ into $F_{\gr}$, corresponding to the old face. So we now assume $n(f)>0$. The tuple $\partial_*(f)=(s_1e_1,\dots,s_le_l)$ describes the path $\partial\varphi_f$ in terms of edges of $E$. If we look for a similar description using instead the elements of $E_{\gr}$, we get the tuple $\partial^{\gr}_*(f)=s_1\gr(e_1)^\frown\dots^\frown s_l\gr(e_l)$
obtained as the concatenation of tuples $s_i\gr(e_i)$, where
\begin{align*}
    +\gr(e)&=(+\gr(e)_1,\dots,+\gr(e)_{n(e)+1}),\\
    -\gr(e)&=(-\gr(e)_{n(e)+1},\dots,-\gr(e)_1).
\end{align*}
Let $\partial^{\gr}_*(f)=(s'_1e'_1,\dots,s'_{L}e'_{L})$. Of these $L=l(f)+n(f)$ indices, exactly $n(f)$ of them $i_1<\dots<i_{n_f}$ have the property that $\partial_0(s'_ie'_i)\in V'\setminus V$ is a new vertex. Let $B_j=\partial^{\gr}_*(f)[i_{j-1}\le i<i_j]$ for $j=1,\dots,n_f$ (wrapping around $\partial^{\gr}_*(f)$ for $j=1$). $B_j$ is a description in terms of edges in $E_{\gr}$ of the section of path $\partial\varphi_f$ located just between the $(j-1)$-th and the $j$-th new vertex met.

We can now insert into $E_{\gr}$ also the $n(f)/2$ edges inside $f$ coming from the curve $\calD$. Suppose that in the Dyck word $W_f$, two matching parentheses have positions $j_0<j_1\in[n(f)]$. If $e'$ is the corresponding edge then we can set $\partial_0(e')=\partial_0(s'_{i_{j_0}}e'_{i_{j_0}})$ and $\partial_1(e')=\partial_0(s'_{i_{j_1}}e'_{i_{j_1}})$. Moreover we define the list $A\in E_{\gr}^{n(f)}$ by setting $A[j_0]=A[j_1]=e'$.

Algorithm \ref{alg:addF} produces the new faces in $F_{\gr}$ inside $f$. The notation $^\pm$ stands for ``(''$^\pm=+$ and ``)''$^\pm=-$.

\begin{algorithm}
\caption{Filling $F_{\gr}$ and the values of $\partial_*\colon F_{\gr}\to(\{+,-\}\times E_{\gr})^{<\infty}$}\label{alg:addF}
\begin{algorithmic}
\State $\text{\textit{unvisited}} \gets [n(f)]$
\While{\textit{unvisited} is not empty}
    \State insert a new element $f$ into $F_{\gr}$
    \State $\partial_*(f) \gets []$
    \State $\text{\textit{depth}} \gets 0$
    \State $\text{\textit{position}} \gets \text{first element of \textit{unvisited}}$
    \While{$\text{\textit{position}} \le n(f)$}
        \If{$\text{\textit{depth}} = 0$}
            \State remove \textit{position} from \textit{unvisited}
            \State insert each element of $B[\text{\textit{position}}]$ into $\partial_*(f)$
            \State insert $W_f[\text{\textit{position}}]^\pm A[\text{\textit{position}}]$ into $\partial_*(f)$
            \If{$W_f[\text{\textit{position}}]=\text{``)''}$}
                \State \textbf{break} \Comment{matching parenthesis found}
            \EndIf
        \EndIf
        \State $\text{\textit{depth}} \gets \text{\textit{depth}} + W_f[\text{\textit{position}}]^\pm1$
        \State $\text{\textit{position}} \gets \text{\textit{position}} + 1$
    \EndWhile
\EndWhile
\end{algorithmic}
\end{algorithm}

\begin{example}
    Figure \ref{fig:insert} illustrates an example of refinement of a face homeomorphic to a Klein bottle. Using the notation of the above algorithm, we have
    \begin{gather*}
        \partial_*(f)=(+e_1,+e_2,+e_1,-e_2),\\
        \gr(e_1)=(e_1',e_2'),\qquad\gr(e_2)=(e_3',e_4',e_5'),\\
        \partial_*^{\gr}(f)=(+e_1',+e_2',+e_3',+e_4',+e_5',+e_1',+e_2',-e_5',-e_4',-e_3'),\\
        B_1=(-e_3',+e_1'),\quad B_2=(+e_2',+e_3'),\quad B_3=(+e_4'),\\
        \quad B_4=(+e_5',+e_1'),\quad B_5=(+e_2',-e_5'),\quad B_6=(-e_4'),\\
        W_f="()()()",\quad A=(e_6',e_6',e_7',e_7',e_8',e_8'),\\
        \partial_*(f_1')=(-e_3',+e_1',+e_6',+e_4',+e_7',+e_2',-e_5',+e_8'),\\
        \partial_*(f_2')=(+e_2',+e_3',-e_6'),\quad\partial_*(f_3')=(+e_5',+e_1',-e_7'),\quad\partial_*(f_4')=(-e_4',-e_8').
    \end{gather*}
    \begin{figure}
        \centering
        \includegraphics[width=0.8\linewidth]{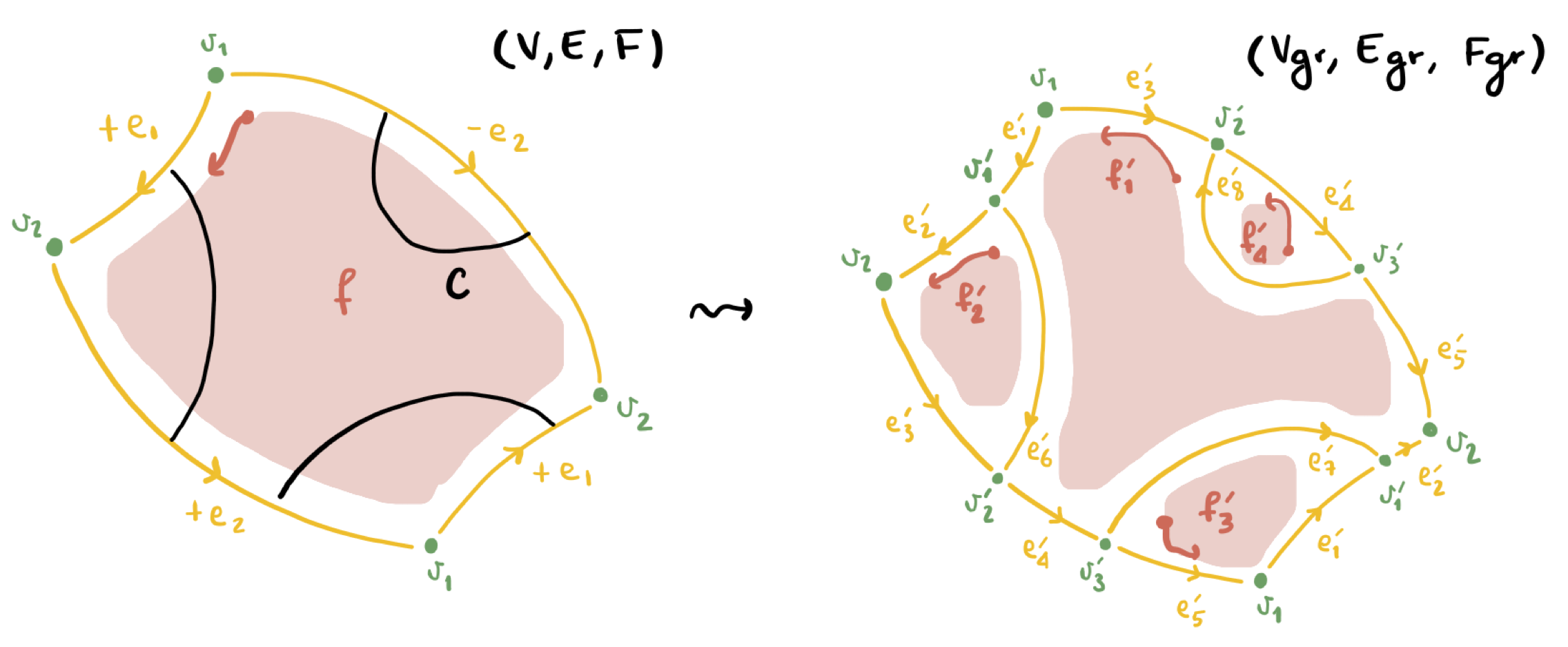}
        \caption{An example of refinement.}
        \label{fig:insert}
    \end{figure}
\end{example}

\subsection{Graphs of a combinatorial curve}\label{subsec:graphs}

Given a combinatorial curve $(n,W,T)$ transverse to $(V,E,F)$, let $(V_{\gr},E_{\gr},F_{\gr})$ be the $(n,W,T)$-refinement of $(V,E,F)$. Suppose that $E$ is also equipped with an index function $i\colon E\to\bbN$; this can be extended to $i\colon E_{\gr}\to\bbN$ by setting $i(e)$ to a value $i'\notin i(E)$ for all edges $e$ coming from the added curve. We define the following graphs (possibly with loops):
\begin{itemize}
    \item the \textbf{ground graph} $\gr(G)$ is a collection of graphs whose vertices are the ground faces $F_{\gr}$, and such that two vertices are joined by an edge in $\gr(G)_i$ if they share an edge of index $i$ in their boundary;
    \item the \textbf{floating forest} $\fl(G)$ is the graph whose vertices are the regions of $\calD\pitchfork(\calV,\calE,\calF)$, and whose edges are given by the union of all floating trees $T_g$;
    \item the \textbf{region graph} $G$ is a collection of graphs whose vertices are the regions of $\calD\pitchfork(\calV,\calE,\calF)$. Two vertices corresponding to ground regions are joined by an edge in $G_i$ if the two ground faces containing them are joined by an edge in $\gr(G)_i$. Additionally, any edge in $\fl(G)$ is also an edge in $G_{i'}$, the region graph whose edges correspond to the added curve.
\end{itemize}

Since any ground face contains a unique region face, we can identify the set of ground faces with the set of ground regions, and regard the ground faces as a subset of all the regions, thus seeing the ground graph $\gr(G)_i$ as a subgraph of the region graph $G_i$ for all $i$.

\begin{figure}
    \centering
    \includegraphics[width=0.7\linewidth]{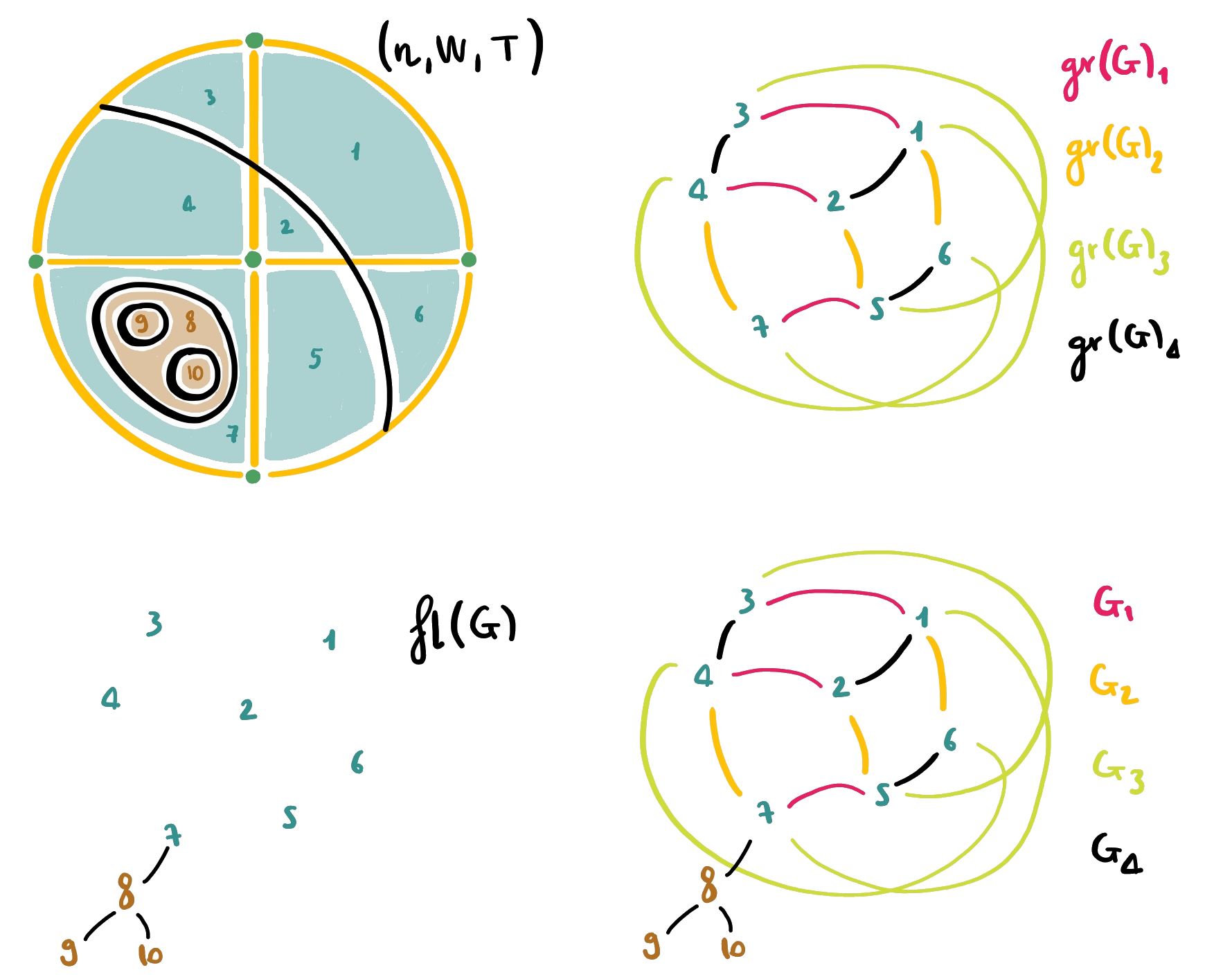}
    \caption{An example of ground, floating and region graphs of a combinatorial curve.}
    \label{fig:graphs}
\end{figure}

\subsection{Removal of edges}\label{subsec:removal}

Let $(\calV,\calE,\calF)$ be a cellular decomposition with an index function $i\colon\calE\to\{0,1\}$. Let $\calD\pitchfork(\calV,\calE,\calF)$ and extend $i\colon\calE_{\gr}\to\{0,1,2\}$ by setting $i(e)=2$ for every $e\subset\calD$. Suppose we are interested in removing from $(\calV,\calE,\calF)$ every edge of index 0: let $\calE'$ be the remaining ones. Set $\calV':=\calV\cap\overline{\bigcup\calE'}$, and $\calF'$ as the connected components of $\calS\setminus\overline{\bigcup\calE'}$. The tuple $(\calV',\calE',\calF')$ might or might not still be cellular: for the following, suppose they are, and that we have chosen characteristic maps.
Then, since $\calD$ is transverse to $(\calV,\calE,\calF)$, obviously it is also transverse to $(\calV',\calE',\calF')$. If $\calD$ induces the combinatorial curve $(n,W,T)\pitchfork(V,E,F)$ and $(n',W',T')\pitchfork(V',E',F')$, we show that $(n',W',T')$ can be calculated without reference to the underlying geometrical objects.

Before describing this procedure, consider the families of triangles $\{\Delta_i(f)\}_{(f,i)}$ and $\{\Delta'_{i'}(f')\}_{(f',i')}$ coming resp.\ from $(\calV,\calE,\calF)$ and $(\calV',\calE',\calF')$ as defined as in the beginning of §\ref{sec:geocom}, and recall that $\overline{\Delta_i(f)}\supset \partial_i(f)$ for all $(f,i)\in\sum_{f\in F}[l(f)]$. Define
\[g\colon\sum_{f\in F}[l(f)]\to F_{\gr}^{<\infty},\qquad g'\colon\sum_{f'\in F'}[l(f')]\to (F'_{\gr})^{<\infty}\]
as follows: for any $e\in E$, let $\gr(e)\in E_{\gr}^{<\infty}$ as in §\ref{subsec:refinement}. If $\partial_i(f)=se$ and $s\gr(e)=(se'_1,\dots,se'_{n(e)+1})$ (remember that $s=-$ reverses the tuple), we let $g(f,i):=(g_1,\dots,g_{n(e)+1})$ where $g_j\in F_{\gr}$ is the unique ground face such that $\overline{\Delta_i(f)\cap g_j}\supset e'_j$. Alternatively, $g(f,i)$ is the sequence of ground faces $g$ whose preimage $\overline{\varphi_f^{-1}(g)}$ is met while traveling along the $i$-th section of the border $\partial\varphi_{\bbD^2}$. Define $g'$ in a similar way.

Define
\[\psi\colon\sum_{f'\in F'}[l(f')]\to\{+,-\}\times\sum_{f\in F}[l(f)]\]
as follows: given any $(f',i')\in\sum_{f'\in F'}[l(f')]$, let $(f,i)\in\sum_{f\in F}[l(f)]$ be the unique pair such that $\overline{\Delta_{i'}(f')\cap\Delta_i(f)}\supset\partial_{i'}(f')$. Then $\partial_{i'}(f')$ and $\partial_i(f)$ refer to the same edge, but might have equal or opposite signs: we let $\psi(f',i')=+(f,i)$ if $\partial_{i'}(f')=\partial_i(f)$ and $\psi(f',i')=-(f,i)$ if $\partial_{i'}(f')=-\partial_i(f)$.


Notice that the data $(n,W,T)$ is essentially equivalent to $(n,g,T)$: in particular we show how to obtain $W_f$ from the tuple $g(f,\cdot)$ for any $f\in F$. To do this, start from the empty word and consider one by one each pair of consecutive entries $(g_i,g_{i+1})$ in the tuple $g(f,1)^\frown\dots^\frown g(f,l(f))$: if $g_i=g_{i+1}$, ignore the pair. If $g_i\neq g_{i+1}$ and the pair $(g_{i+1},g_{i})$ hasn't yet occurred, insert a ``('' into the word, otherwise insert ``)''.

Unlike $g$, the function $\psi$ cannot be derived from the data $(V,E,F)$, $(V',E',F')$, $(n,W,T)$ alone and has to be explicitly provided in the following procedure to resolve any ambiguities.

We begin our construction by considering the quotient $\pi\colon F_{\gr}\to\pi(F_{\gr})$ obtained by contracting the edges of the ground graph $\gr(G)_0$. Then
\[g'(f',i')=\pi(g(\psi(f',i')))\]
where $g(-(f,i))$ is the reverse of the tuple $g(+(f,i))=g(f,i)$. From this we immediately recover the Dyck word $W'_{f'}$ for every $f'\in F'$. Furthermore, $F_{\gr}'$ can be identified as the set of all elements $\pi(F_{\gr})$ appearing as the component of some $g'(f',i')$, while any element of $\pi(F_{\gr})\setminus F'_{\gr}$ represents a class of ground regions that after the removal of the edges have become a floating region.

To finally obtain the floating forest of $\fl(G')$ of $(n',W',T')$, first consider the region graph $G_2$ whose edges represents edges of the curve $\calD$, and let $\pi(G_2)$ be the contraction of those vertices which are linked by an edge in $G_0$ similarly to before. Define $H$ as the subgraph of $\pi(G_2)$ with the same set of vertices, and where every edge between two elements of $F'_{\gr}$ is removed. Then
\[\fl(G')=\pi(\fl(G))\cup H,\]
from which we recover $(n',W',T')\pitchfork(V',E',F')$.

Notice that after the removal, some vertices in $\calV'$ might be in the boundary of just two edges of the same type; these vertices can be removed and the pairs of edges can be merged, thus obtaining yet another cellular decomposition $(\calV'',\calE'',\calF')$ where the set of faces remains unchanged. It is easy to recover the corresponding combinatorial curve $(n'',W',T')\pitchfork(V'',E'',F')$, the only difference being that if edges $e'_1,\dots,e'_l\in\calV'$ merge into $e''\in\calV''$ then $n_{e''}=\sum_{i=1}^ln(e'_i)$.

\begin{example}
    \begin{figure}
        \centering
        \includegraphics[width=0.7\linewidth]{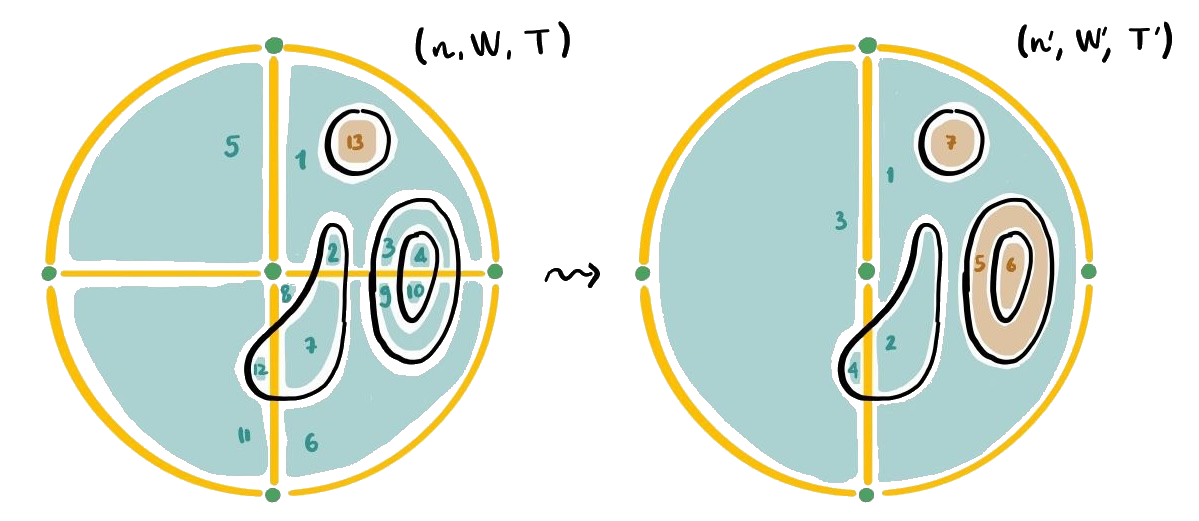}
        \caption{An example of removal of edges.}
        \label{fig:remove}
    \end{figure}

    In Figure \ref{fig:remove} we have an example of removal of the $\{y=0\}$-axis. In this case we have
    \begin{gather*}
        g(1,1)=(1),\quad g(1,2)=(1,2,1,3,4,3,1),\quad g(1,3)=(1),\\
        g(2,1)=g(2,2)=g(2,3)=(5),\\
        g(3,1)=(6,7,8),\quad g(3,2)=(8,7,6,9,10,9,6),\quad g(3,3)=(6),\\
        g(4,1)=(11,12,11),\quad g(4,2)=g(4,3)=(11),\\
        g'(1,1)=(1,2,1),\quad g'(1,2)=(1),\quad g'(2,1)=(3,4,3),\quad g'(2,2)=(3),\\
        \psi(1,1)=+(1,1),\;\psi(1,2)=-(3,1)\;\psi(1,3)=-(3,3),\;\psi(1,4)=+(1,3),\\
        \psi(2,1)=+(2,1),\;\psi(2,2)=-(4,1),\;\psi(2,3)=-(4,3),\;\psi(2,4)=+(2,3).
    \end{gather*}
\end{example}

\section{Obstructions to realizability}\label{sec:obstructions}

\subsection{The curve type of an arrangement} \label{subsec:curvetype}

It is well-known \cite{rokhlin1978complex} that the topological type of a curve in $\bbP^2(\bbR)$ is determined by two pieces of information:
\begin{itemize}
    \item whether the curve has a \textbf{pseudoline} (that is, a connected component which does not disconnect $\bbP^2(\bbR)$) or not;
    \item the rooted tree formed by the regions determined by the curve, where the root is the unique region whose closure is non-orientable, which we call the \textbf{external region}.
\end{itemize}
Let $\calD$ be a curve of degree $d$ transverse to a cellular arrangement $(\calC_1=\calL,\calC_2,\dots,\calC_k)$ where $\calL$ is a line. Let $i\colon\calE_{\gr}\to[k+1]$ be the index function where $i(e)=i$ for every $e\subset\calC_i$ and $i(e)=k+1$ for $e\subset\calD$. We want to determine the topological type of $\calD$, forgetting the cellular arrangement, from just the combinatorial data $(n,W,T)$. This is called the \textbf{curve type} of $(n,W,T)$.

First suppose $\calD$ has a pseudoline. Take the region graph $G_{k+1}$ whose edges represent regions separated by $\calD$, and merge any two regions separated by any other curve $\calC_i$ (that is, any two vertices joined by any edge in $G_i$ for $i=1,\dots,k$) obtaining a graph $G'_{k+1}$. Then, the vertices of $G'_{k+1}$ correspond to the regions determined by $\calD$, and the presence of the pseudoline ensures that the external region has a self-loop. Thus in this case we have determined $\calD$ up to topological equivalence.

Suppose now on the contrary that $\calD$ has no pseudoline. Consider the region graphs $G_1,G_{k+1}$ and merge any two regions separated by a curve $\calC_i$ for just $i=2,\dots,k$, obtaining two graph $G'_1,G'_{k+1}$. Since $\calD$ has no pseudoline, it is easy to see that the arrangement $(\calL,\calD)$ admits an external region, separated from itself by the line $\calL$, which determines a unique vertex $v'_{\mathrm{ext}}\in G'_1$ with a self-loop. Now starting from $G'_{k+1}$ merge any two vertices separated by an edge in $G'_1$, obtaining the quotient graph $\pi\colon G'_{k+1}\to G''_{k+1}$ whose points represent the regions determined by $\calD$. Then, the external region is the image $\pi(v_{\textrm{ext}})$ of the external region of $(\calL,\calD)$ under the quotient formed by forgetting the line $\calL$.

Of course, a priori we do not know whether $\calD$ has a pseudoline or not, so we first check whether $(\calL,\calD)$ has an external region, i.e.\ whether $G_1'$ has a self-loop or not, and then follow the first or the second procedure accordingly.

\subsection{Bézout's criterion of order 0}\label{subsec:bezout0}

The real version of Bézout's theorem states that two nonsingular algebraic curves $\calC,\calC'$ on $\bbP^2(\bbR)$ of degree resp.\ $d,d'$ transverse with each other intersect in $\#(\calC\cap\calC')\le dd'$ points, and
\[dd'\equiv\#(\calC\cap\calC')\pmod{2}.\]
Suppose we are given a $(1,d_2,\dots,d_k)$-arrangement $(\calC_1=\calL,\dots,\calC_k)$ inducing the combinatorial cell structure $(V,E,F)$, where $\calL$ is a line. The Bézout's theorem gives us multiple criteria for the $d_{k+1}$-realizability of a combinatorial curve $(n,W,T)\pitchfork(V,E,F)$.

In the following, let $d:=(d_1,\dots,d_{k+1})$ and if $(n,W,T)$ is $d_{k+1}$-realizable let $\calC_{k+1}$ be one realization.

\begin{definition}
    We say that $(n,W,T)$ \textbf{respects basic $d$-Bézout} if for any $i=1,\dots,k$ the sum $\sum_{i(e)=i}n_e$ is a number $\le d_id_{k+1}$ with the same parity as $d_id_{k+1}$.
\end{definition}

Respecting basic Bézout is a necessary criterion for the $d$-realizability of $(n,W,T)$ which puts a bound on the numbers $n_e$. These in turn determine the length of the Dyck words $W_f$. A simple way to bound the size of the floating trees $T_g$ is Harnack's inequality, which states that a curve of degree $d$ can have at most $(d-1)(d-2)/2+1$ ovals; using the known results on Hilbert's 16th problem for curves of degree $1\le d\le 8$ together with our algorithm to determine the topological type of a curve, we get a stronger criterion. As a consequence, fixed $(\calC_1,\dots,\calC_k)$, we can effectively list all possible candidate $d$-realizable combinatorial curves $(n,W,T)\pitchfork(V,E,F)$, and in particular the set of $d$-realizable combinatorial curves transverse to $(V,E,F)$ is finite.

We can find stronger consequences of Bézout's theorem. Let $G$ be the region graph of $(n,W,T)\pitchfork(V,E,F)$. A \textbf{combinatorial curve} $l$ with intersection numbers $(l_1,\dots,l_{k+1})$ is a closed walk on $G$ passing through exactly $l_i$ edges of $G_i$ for all $i=1,\dots,k+1$, such that for all $i\in[k+1]$, $l_i\le d_i$ and $l_i$ has the same parity as $d_i$. Since $d_1=1$, necessarily $l_1=1$: this guarantees that, if realized by an actual connected curve, $l$ would be a pseudoline.

\begin{definition}
    A combinatorial curve $(n,W,T)$ \textbf{respects $d$-Bézout at order 0} if for any two regions $r,r'\in G$ there exists a combinatorial line passing through them.
\end{definition}

\begin{figure}
    \centering
    \includegraphics[width=0.25\linewidth]{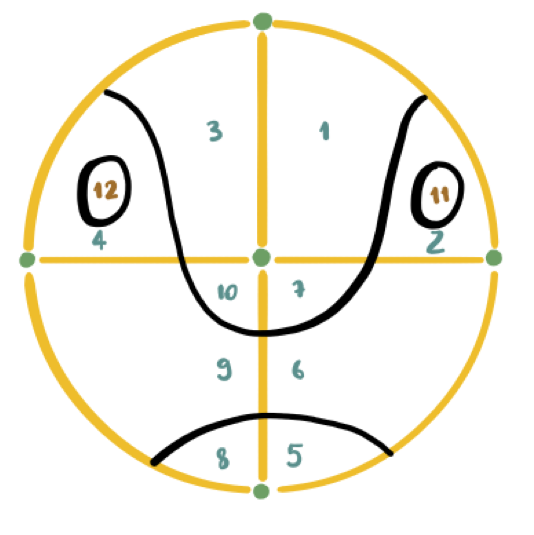}
    \caption{A curve satisfying basic Bézout but not Bézout at order 0 for $d=(1,1,1,4)$. Note that there is a closed walk through regions 11 and 12 crossing the curve 4 times, but no such walk which also crosses each axis exactly once.}
    \label{fig:placeholder}
\end{figure}

Respecting $d$-Bézout at order 0 is a necessary condition for $d$-realizability, since for any two regions determined by $\calD\pitchfork(\calC_1,\dots,\calC_k)$ we can obviously find an actual line $\calL$ transverse to $(\calC_1,\dots,\calC_k,\calD)$ going through those regions, which translates to the existence of a combinatorial line satisfying the above requirements on its intersection numbers, by Bézout's theorem.

We now discuss an effective algorithm to check whether a combinatorial curve respects Bézout at order 0 or not. Similar techniques are standard in computer science and graph theory \cite{maclagan2015introduction, gondran2008graphs}. Let $G$ be the region graph and identify the set of its vertices with $[\rho]$. Letting
\[P:=\{0,\dots,d_1\}\times\cdots\times\{0,\dots,d_{k+1}\},\]
we can give its powerset $\calP:=2^P$ the structure of a semiring by letting
\begin{align*}
    0_{\calP}&:=\varnothing,\\
    1_{\calP}&:=\{(0,\dots,0)\},\\
    s+_{\calP}s'&:=s\cup s',\\
    s\cdot_{\calP}s'&=\{(p_1+p'_1,\dots,p_{k+1}+p'_{k+1})\mid p\in s,p'\in s',p_i+p'_i\le d_i\;\forall i\in[k+1]\}.
\end{align*}
Construct the symmetric matrix $A\in\calP^{\rho\times\rho}$ where for all $r,r'\in[\rho]$,
\[A_{r,r'}=\{e_i\mid\text{$(r,r')$ is an edge of $G_i$}\}\]
where $e_i$ are the unit vectors in $\bbZ^{k+1}$. Let $A^i=A\cdot_{\calP}\ldots\cdot_{\calP}A$ be the matrix multiplication using our semiring structure. If $\delta:=d_1+\dots+d_{k+1}$ then
\[B:=\sum_{i=0}^\infty A^i=I_{\calP}+A+A^2+\dots+A^\delta\in\calP^{\rho\times\rho}\]
is a symmetric matrix where each entry $B_{r,r'}$ is the set of all possible tuples $p=(p_1,\dots,p_{k+1})\in P$ such that there exists a walk in $G$ from $r$ to $r'$ passing through $p_1$ edges of index 1, $p_2$ edges of index 2 etc.

\begin{proposition}
    Suppose that $(n,W,T)$ respects basic $d$-Bézout. Then it also respects $d$-Bézout at order 0 if and only if for all $r,r'\in[\rho]$ there exists a tuple $(p_1,\dots,p_{k+1})\in B_{r,r'}\cdot_{\calP}B_{r,r'}$ such that $p_i$ has the same parity as $d_i$ for all $i\in[k+1]$.
\end{proposition}

\begin{proof}
    There exists a tuple $(p_1,\dots,p_{k+1})\in B_{r,r'}\cdot_{\calP}B_{r,r'}$ if and only if there exist tuples $p',p''\in B_{r,r'}$ with $p'+p''=p$, which means that there is a walk in $G$ from $r$ to $r'$ passing through $p'_i$ edges of $G_i$ and a walk from $r'$ to $r$ passing through $p''_i$ edges of $G_i$: by joining them we get the required combinatorial line.
\end{proof}


\subsection{Bézout's criterion of order \texorpdfstring{$n$}{n}}\label{subsec:bezoutn}

The following is a family of Bézout criteria defined for floatless curves. Consider as before a $(1,d_2,\dots,d_k)$-arrangement $(\calC_1=\calL,\dots,\calC_k)$ inducing $(V,E,F)$, and a combinatorial curve $(n,W,T)\pitchfork(V,E,F)$ which respects $(d_1,\dots,d_{k+1})$-Bézout at order 0, and for which $T_g$ is the trivial tree for all $g$, i.e.\ $(n,W,T)$ is floatless. Then $(n,W,T)$ is fully characterized by the combinatorial cell complex $(V_{\gr},E_{\gr},F_{\gr})$. Let $l$ be a combinatorial line on $(n,W,T)$. A \textbf{realization} of $l$ is a floatless combinatorial curve $(n^{(l)},W^{(l)},T^{(l)})$ transverse to $(V_{\gr},E_{\gr},F_{\gr})$ which has the curve type of a line (so, has a pseudoline and only the external region) and that agrees with $l$ when seen as a walk on the graph of $(V_{\gr},E_{\gr},F_{\gr})$. Concretely, we can find all such realizations by determining from $l$ the values $n^{(l)}$, generating all possible Dyck words and filtering out those that have the wrong curve type or do not agree with $l$.

\begin{definition}
    For any $n$ positive integer, we say that $(n,W,T)$ \textbf{respects $d$-Bézout at order $n$} if for all pairs of regions $r,r'\in F_{\gr}$ there exists a $d$-combinatorial line $l$ through $r,r'$ and a realization $(n^{(l)},W^{(l)},T^{(l)})\pitchfork(V_{\gr},E_{\gr},F_{\gr})$ that respects $(d_1,\dots,d_{k+1},1)$-Bézout at order $n-1$.
\end{definition}

\begin{example}
    \begin{figure}
        \centering
        \includegraphics[width=0.7\linewidth]{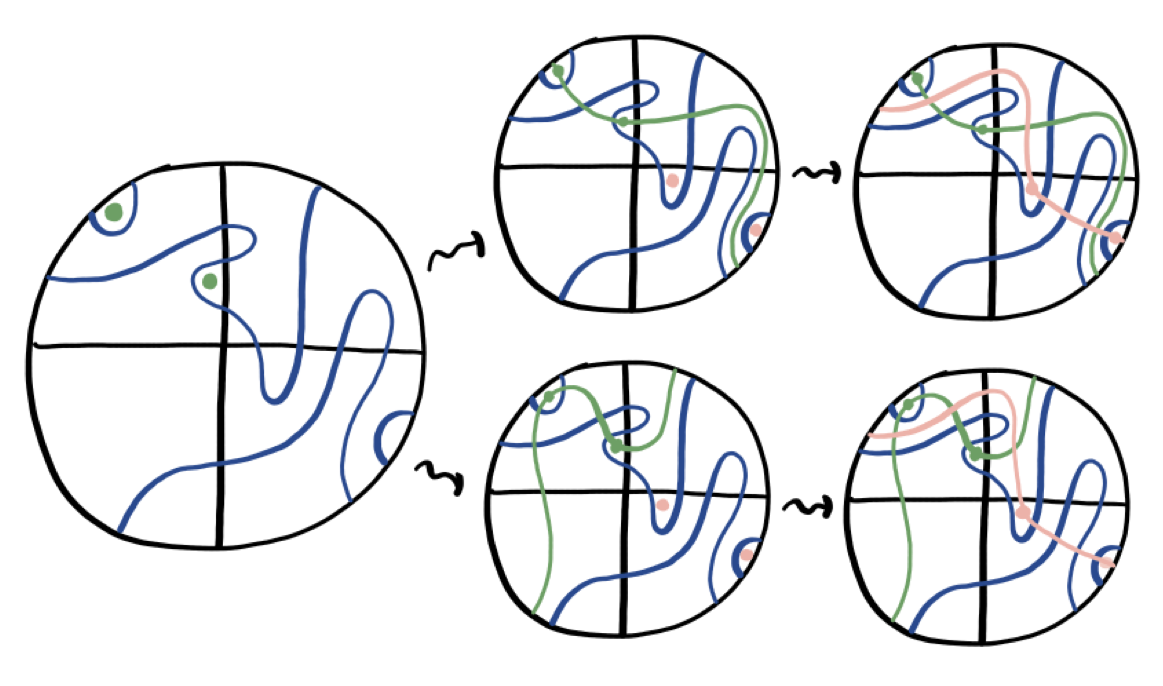}
        \caption{An example of combinatorial curve that respects Bézout at order 0 but not at order 1.}
        \label{fig:bezout1}
    \end{figure}
    The combinatorial curve in \ref{fig:bezout1} does not respect Bézout at order 1: although there exist combinatorial lines both through the pair of pink and green dots, we cannot simultaneously choose for both pairs combinatorial lines that meet only once.
\end{example}

\begin{definition}
    The combinatorial curve $(n,W,T)$ \textbf{admires $d$-Bézout} if it respects $d$-Bézout at order $n$ for all $n$ positive integers.
\end{definition}

Clearly, Bézout at order $n$ and admiration are necessary conditions for realizability. There are some natural questions related to admiration:
\begin{enumerate}
    \item are there non-$d$-realizable $d$-admirers of Bézout?
    \item is there a finite procedure to check $d$-admiration?
\end{enumerate}

We implemented an algorithm that checks Bézout at order $n$ for floatless arrangements. In practice, we used Bézout at order 1 in our analysis, but even checking Bézout at order 2 is too slow to be of any use. This makes it hard to study higher-order Bézout criteria and admiration at the moment.

Also, an interesting problem is to find an effective way to check Bézout at orders $>0$ for non-floatless arrangements: the most problematic step is finding all the realizations of a combinatorial line, since the presence of floating ovals increases the complexity considerably.

\section{Construction of realizable cases}\label{sec:constructions}

\subsection{Viro's patchworking construction}\label{subsec:viro}

The main tool we employ to generate realizable cases is Viro's Patchworking Construction \cite{viro2006patchworking}. What is striking is that even by taking into consideration the interactions of the curve with the three lines $\calL_x,\calL_y,\calL_z$, no adjustments have to be made to the original construction by Viro.

In this section we follow the treatment given in \cite{gelfand1994discriminants}[11, §5] (in particular, no unimodularity assumptions are needed). We notice that through the theory of combinatorial cell complexes and arrangements developed in the previous sections, Viro's construction is particularly natural to formulate: consider the polytope
\[Q_d=\conv\{(0,0),(d,0),(0,d)\}\subset\bbR^2\]
and let $A=Q_d\cap\bbZ^2$ by its integer points. A \textbf{triangulation} $(\calV,\calE,\calF)$ of $(Q_d,A)$ is a cellular decomposition of $Q_d$ into triangles (with straight segments as edges, as opposed to ``topological'' triangles) such that $\calV\subseteq A$, i.e.\ every vertex belongs to $A$. A \textbf{regular triangulation} is a triangulation such that there exists a height function $h\colon\calV\to\bbR$ with the property that by taking the convex hull of the graph $\Gamma_h=\{(x,y,z)\in\bbR^3\mid(x,y)\in\calV,z=h(x,y)\}$ and projecting down through $(x,y,z)\mapsto(x,y)$ the faces of the hull, we recover exactly the original triangles in $\calF$.

Suppose we are given a regular triangulation $\calT=(\calV,\calE,\calF)$ of $(Q_d,A)$, with vertices $\calV$ and height function $h\colon\calV\to\bbR$. Choose an arbitrary sign function $\varepsilon\colon\calV\to\{+,-\}$ and let
\[n_e=\begin{cases}1&\text{if $\varepsilon(\partial_0(e))\neq\varepsilon(\partial_1(e))$,}\\0&\text{otherwise}.\end{cases}\]
Since each triangle $f\in\calF$ has either 0 or 2 edges $e$ with $n_e=1$, the only possible choice of Dyck word $W_f$ is resp.\ the empty word or ``()''. If we now choose the trivial rooted tree for every ground face $g\in\calF_{\gr}$ we get a combinatorial curve $(n,W,T)$ transverse to the combinatorial cell complex $(V,E,F)$ induced by $(\calV,\calE,\calF)$. Now let $i\colon E\to\{0,1\}$ be the index function such that $i(e)=1$ if $e\subset\partial Q_d$ is on the boundary of the triangle and $i(e)=0$ otherwise. By removing the edges of index 0 as in §\ref{subsec:removal} we obtain a combinatorial curve $(n_{xyz},W_{xyz},T_{xyz})$ transverse to the triangle $Q_d$.

Finally, let $\varepsilon_{xYZ}(x,y)=(-1)^x\varepsilon$, $\varepsilon_{XyZ}(x,y)=(-1)^y\varepsilon$ and $\varepsilon_{XYz}(x,y)=(-1)^{x+y}\varepsilon$, and construct resp.\ $(n_{xYZ},W_{xYZ},T_{xYZ})$, $(n_{XyZ},W_{XyZ},T_{XyZ})$ and $(n_{XYz},W_{XYz},T_{XYz})$ in the same way. Then by interpreting these as combinatorial curves on the faces $xyz,xYZ,XyZ,XYz$ of $(V^{(3)},E^{(3)},F^{(3)})$ we obtain a combinatorial curve which we name $C(\calT,\varepsilon)$.

\begin{theorem}[Viro]
    The combinatorial curve $C(\calT,\varepsilon)$ is $(1,1,1,d)$-realizable from the polynomial
\[f_t(x,y,z)=\sum_{(a,b)\in\calV}\varepsilon(a,b)t^{h(a,b)}x^ay^bz^{d-a-b}\]
for any $t>0$ sufficiently small.
\end{theorem}
\begin{proof}
    Let $\calC_t=\{f_t=0\}$. As stated in \cite{gelfand1994discriminants}[Page 382], for $t>0$ sufficiently small, the curve $\calC_t$ is nonsingular, intersects the axes $xyz=0$ transversely, and its isotopy type in the space of nonsingular curves of degree $d$ transverse to the axes only depends on $(\calT,\varepsilon)$. By \cite{gelfand1994discriminants}[Theorem 11.5.6], the intersection $\calC_t\cap\{x\ge0,y\ge0,z\ge0\}$ has topological type $(n_{xyz},W_{xyz},T_{xyz})$, and similarly for the other faces. Thus, the arrangement $(\calL_x,\calL_y,\calL_z,\calC_t)$ has topological type $C(\calT,\varepsilon)$.
\end{proof}

\begin{figure}
    \centering
    \includegraphics[width=0.7\linewidth]{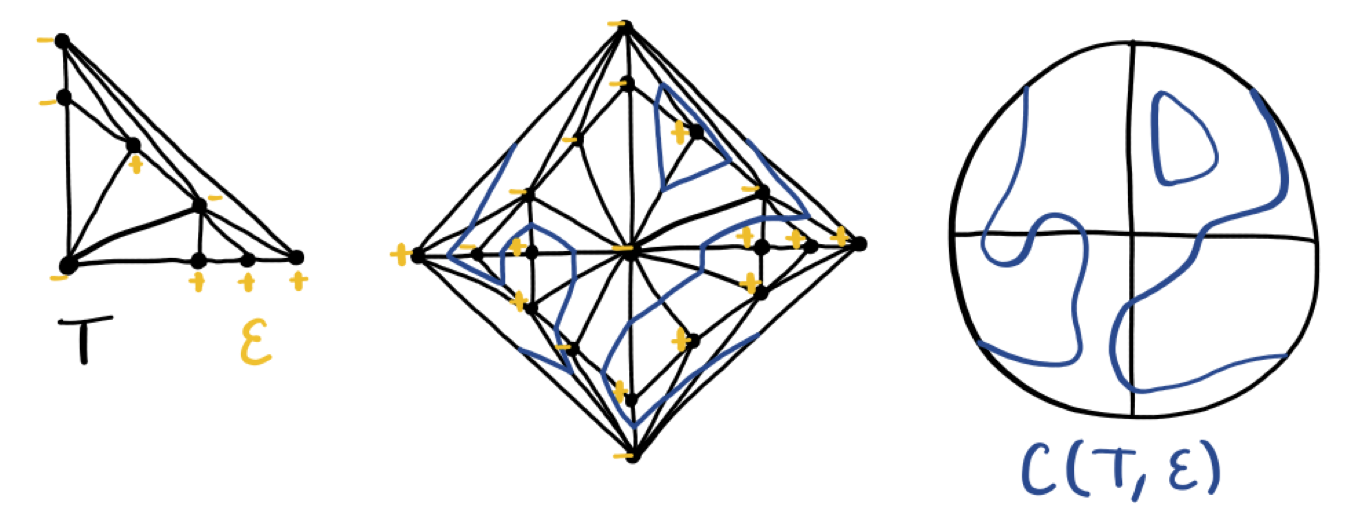}
    \caption{Example of Viro's patchworking.}
    \label{fig:viro}
\end{figure}

\subsection{Translations}\label{sec:translations}

Suppose we are given a degree $d$ curve $\calD\pitchfork(\calL_x,\calL_y,\calL_z)$ with affine equation $f(x,y)=0$. Then a generic vertical translation $f(x,y-y_0)=0$ of $\calD$ gives us a new degree $d$ curve $\calD_{y_0}\pitchfork(\calL_x,\calL_y,\calL_z)$. In order to generate new realizable combinatorial curves from old ones, it would be useful to have a procedure to turn the combinatorial data $D\pitchfork(V^{(3)},E^{(3)},F^{(3)})$ induced by $\calD$ into the combinatorial data $D_{y_0}$ induced by $\calD_{y_0}$. Unfortunately, knowing $D$ is not enough to be able to recover $D_{y_0}$ for any $y_0$: the topological type of the arrangement might change for any critical point of the linear functional $\alpha(x,y)=y$ restricted to $\calD$, and at the intersections $\calD\cap\calL_x$.

However, by the algebraicity of $\calD$ there is only a finite number of such points, so for $y_0\gg0$ large enough the topological type of the upward translation $\calD_{y_0}$ depends uniquely on the topological type of $\calD$. Let $D=(n,W,T)$. Then, to construct $D_{y_0}$ (Figure \ref{fig:translation}),
\begin{enumerate}
    \item remove $L_y$ from the configuration, obtaining $(n',W',T')\pitchfork(V^{(2)},E^{(2)},F^{(2)})$;
    \item insert a line $L_y'$ going through the same ground faces that $L_z$ meet: concretely, the two faces above $L_y'$
    \[W''_{xyz}=W'_{xXZz},\quad W''_{xYZ}=W'_{xXzZ}\]
    while $W''_{XyZ}$ is the word comprised of $n_Z$ consecutive ``('' followed by $n_Z$ ``)'', and similarly $W''_{XYz}$ is made of $n_z$ consecutive ``('' followed by the same number of ``)''. The values $n''$ are set from $n'$ as follows,
        \[n''_{x}=n_x+n_X,\quad n''_y=n_Z,\quad n_z''=n_z,\quad n''_X=0,\quad n''_Y=n_z,\quad n''_Z=n_Z,\]
    while the floating trees $T''$ are inherited from $T'$ on the faces $xyz,xYZ$ and are trivial on $XyZ,XYz$.
\end{enumerate}
The result $(n'',W'',T'')$ is the combinatorial curve $D_{y_0}$ we were looking for. Other translations can be obtained from this one by applying the automorphisms of $(V^{(3)},E^{(3)},F^{(3)})$, which do not affect the algebraicity or the degree of the curve since every such automorphism comes from a projective transformation. This two-step process of using the upward translation and other automorphisms can be repeated multiple times: in practice, we apply it to arrangements generated by Viro's patchworking, repeating it until no new cases are obtained.

\begin{figure}
    \centering
    \includegraphics[width=0.7\linewidth]{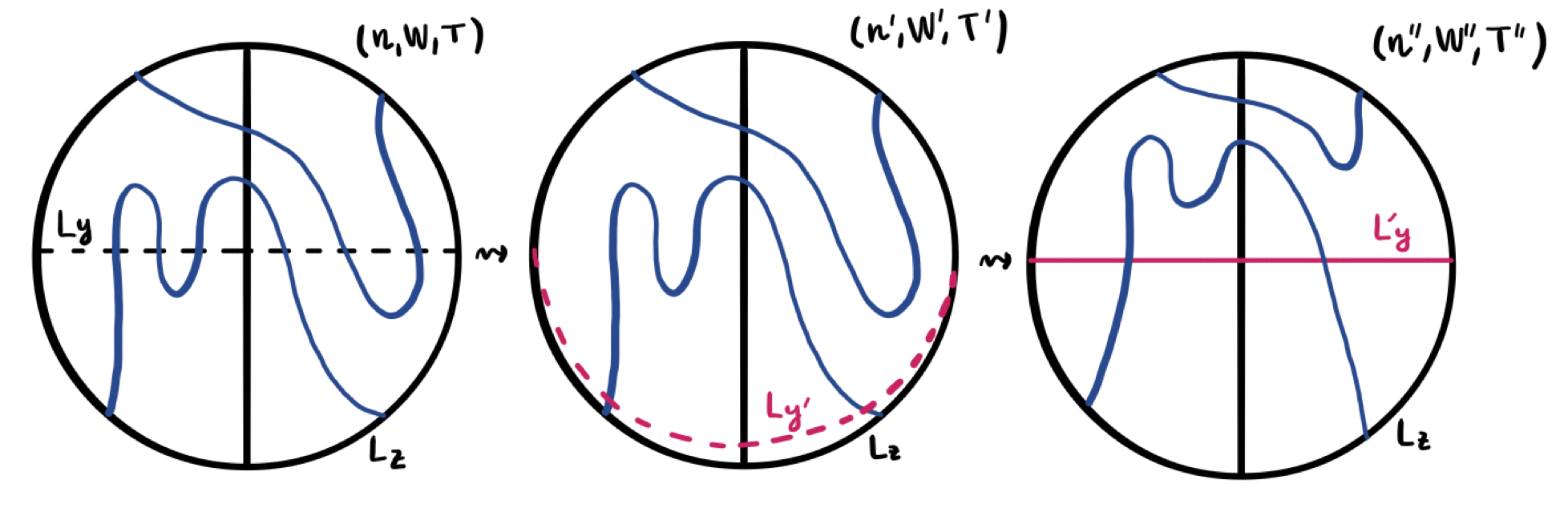}
    \caption{Example of a translation.}
    \label{fig:translation}
\end{figure}

\section{Classification} \label{sec:results}

In the repository associated to this work \cite{NWTrepo}, a full database of the combinatorial curves discussed here is available. In particular, to reproduce the results in this section it is sufficient to run the notebooks \texttt{cubic.ipynb} and \texttt{quartic.ipynb}. All computations use exact arithmetic throughout.

\subsection{Description of the enumeration}

The classification procedure has many steps, which we now describe in the case of arrangements of three lines and a quartic. The case of cubics is similar but simpler.

\paragraph{\textbf{Step 1 - Enumeration of all regular triangulations}} Viro's construction (§\ref{subsec:viro}) takes as one of its inputs any regular triangulation of $\conv\{(0,0),(4,0),(0,4)\}$. To generate as many cases as possible, we start by generating all possible such triangulations. This is carried out in the the \texttt{Macaulay2} \cite{M2} script \texttt{quartic\_triangulations.m2} and uses the \texttt{Topcom} \cite{TopcomSource} package. The symmetry group of the triangle acts on the set of regular triangulations, and in order to reduce the number of cases, only one triangulation per orbit is kept in the database. The output is the database \texttt{quartic\_triangulations.csv} and consists of 97460 triangulations, together with the values of the height functions which certify that the triangulations are regular.

\paragraph{\textbf{Step 2 - Viro's patchworking}} Viro's patchworking (§\ref{subsec:viro}) is applied to every triangulation and every possible sign choice. Symmetries arising from permutations of the faces is leveraged in order to again reduce the number of sign choices considered. In particular, once a sign $\varepsilon$ has been used, the signs relative to the other 4 faces such as $\varepsilon(x,y)=(-1)^{x}\varepsilon(x,y)$ are not considered.

\paragraph{\textbf{Step 3 - Symmetries and translations}} To the resulting combinatorial curves, we first apply the action of the automorphism group $(V^{(3)},E^{(3)},F^{(3)})$ (Example \ref{ex:lines_xyz2}) in order to recover those cases not considered by removing the symmetries in Step 1 and 2. Then we add those obtained by upward translation (§\ref{sec:translations}) and again the action of the automorphism group of $(V^{(3)},E^{(3)},F^{(3)})$; this process is repeated until no new curves emerge, which is after 2 times. These make up the set of realized curves \texttt{viro}, saved in the compressed database \texttt{quartic\_viro.jsonl.gz} together with information on the triangulation they were generated from.

\paragraph{\textbf{Step 4 - Generation of all floatless cases}} To get the admissible curves, which are our candidates for what can be realizable, we first generate all combinatorial curves transverse to $(V^{(3)},E^{(3)},F^{(3)})$ which
\begin{itemize}
    \item are floatless;
    \item have \textbf{quartic curve type}, meaning they have curve type (§\ref{subsec:curvetype}) equal to one of the 6 possible for quartic curves (between 0 and 4 non-nested ovals, or two nested ovals);
    \item respect Bézout at order 1 (§\ref{subsec:bezoutn}).
\end{itemize}
We end up with 40270 curves. This is done by first generating the possible intersection numbers $n$ which satisfy basic Bézout and the rule that each $n_f$ must be even. Then, for each face $f$, we generate all possible Dyck words $W_f$ with lengths compatible with $n$. Finally, from every combination of Dyck words, the combinatorial curves $(n,W,T_0)$ with all trivial floating trees are formed: for each of these, we check whether the curve type is that of a quartic and whether it satisfies Bézout at order 1, removing it if the criteria are not satisfied.

\paragraph{\textbf{Step 5 - Generation of all admissible cases}} We first need a lemma.

\begin{lemma}
    Suppose that $(n,W,T)\pitchfork(V^{(3)},E^{(3)},F^{(3)})$ is a $(1,1,1,4)$-realizable combinatorial curve. Let $(n,W,T_0)$ be the floatless curve obtained by removing all floating ovals. Then $(n,W,T_0)$ has quartic curve type and admires $(1,1,1,4)$-Bézout, so it is among those generated in Step 4.
\end{lemma}
\begin{proof}
    The condition of having between 0 and 4 non-nested ovals or 2 nested ovals is obviously still met after removing ovals. Let $\calC\subset\bbP^2(\bbR)$ be a curve realizing $(n,W,T)$. Then its ground part $\gr(\calC)$ realizes $(n,W,T_0)$. Any line transverse to $\calC$ is also transverse to $\gr(\calC)$ and intersects it in $\le4$ points, so $(n,W,T_0)$ respects Bézout at order 0, and by induction on $n$ respects Bézout at any order $n$.
\end{proof}

Generate all combinatorial curves $(n,W,T)$ transverse to $(V^{(3)},E^{(3)},F^{(3)})$ which
\begin{itemize}
    \item have quartic curve type;
    \item respect Bézout at order 0; and for which
    \item after removing all floating ovals, the floatless curve $(n,W,T_0)$ is among those generated in Step 4.
\end{itemize}
By the lemma, this is a necessary condition for $(1,1,1,4)$-realizability. This the set of admissible curves \texttt{bezout}, stored in the compressed database \texttt{quartic\_bezout.jsonl.gz}. To generate them, start from the set of floatless curves obtained in Step 4, and insert floating ovals in all possible ways that do not make the resulting arrangement have non-quartic type or violate Bézout at order 0.

This is done as follows: from a floatless curve $(n,W,T_0)$, the \textbf{line distance matrix} is calculated, which is a matrix $L_{r,r'}$ such that for every pair of ground regions $g,g'$, the entry $L_{g,g'}$ records the minimum number of intersections between the curve and a combinatorial line going through both $g,g'$. The line distance matrix is calculated using the matrix $B\in\calP^{\rho\times\rho}$ defined in §\ref{subsec:bezout0}. A few remarks are that:
    \begin{itemize}
        \item the fact that Bézout at order 0 is respected is equivalent to the fact that $L_{g,g'}\le4$ for all entries, and we want each entry to be $\le4$ also after adding the floating ovals.
        \item If we insert a floating region $f$ inside a ground region $g$, then $L_{f,g'}=L_{g,g'}+2$ for all ground regions $g'$.
        \item If we insert two floating regions $f,f'$ inside the ground regions resp.\ $g,g'$ (which can coincide), then $L_{f,f'}=4+L_{g,g'}$.
    \end{itemize}
    Therefore, to find all possible ways to add ovals to $(n,W,T_0)$, let
    \[R=\{g\in F_{\gr}\mid L_{g,g'}\le2\text{ for all $g'\in F_{\gr}$}\}\]
    and let $m$ be the maximum number of floating ovals that could be added to the curve, which is 0 if $(n,W,T_0)$ is made up of 2 nested ovals, and $4-c$ if it has $c$ non-nested ovals. Consider any multiset $M$ of size $\le m$ of elements of $R$: in light of the previous remarks, we can insert one floating oval in each element of the multiset and obtain a curve $(n,W,T)$ respecting Bézout at order 0, if and only if for each pair of elements $g,g'\in M$ (which can be equal if $g$ appears twice or more in $M$) we have $L_{g,g'}=0$. We assume that the added floating ovals do not nest with each other: we do not need to systematically consider the insertion of nested floating ovals, since for quartic curves the only way to have two floating ovals nested with each other is for them to be the whole curve, so there are only 4 such cases (one for each face of $(n,W,T_0)$).

\paragraph{\textbf{Step 6 - Final check}} Verify that \texttt{viro} is a subset of \texttt{bezout} as expected, and find out the difference \texttt{rest}.

The correctness of this algorithm proves that the ``realized'' elements of \texttt{viro} are indeed realizable, and that all realizable cases are among the ``admissible'' elements \texttt{bezout}.

For the following results, recall (Example \ref{ex:unlabeled}) that the combinatorial curves up to $i$-automorphisms are the topological types of labeled arrangements $(\calL_x,\calL_y,\calL_z,\calC)$, while the combinatorial curves up to automorphisms are the topological types of unlabeled arrangements allowing for permutations of the three lines.

\subsection{Up to three lines and a cubic curve}\label{subsec:cubic}

We first look at our results concerning configurations of three lines and a cubic curve. In this case Viro's patchworking and translations produce 2024 distinct combinatorial curves, while our criteria select 2072 admissible arrangements. Their difference up to automorphism consists of just 3 curves, shown in Figure \ref{fig:rest3}. 

\begin{figure*}[t!]
    \centering
    \begin{subfigure}[t]{0.18\textwidth}
        \centering
        \includegraphics[width=0.9\linewidth]{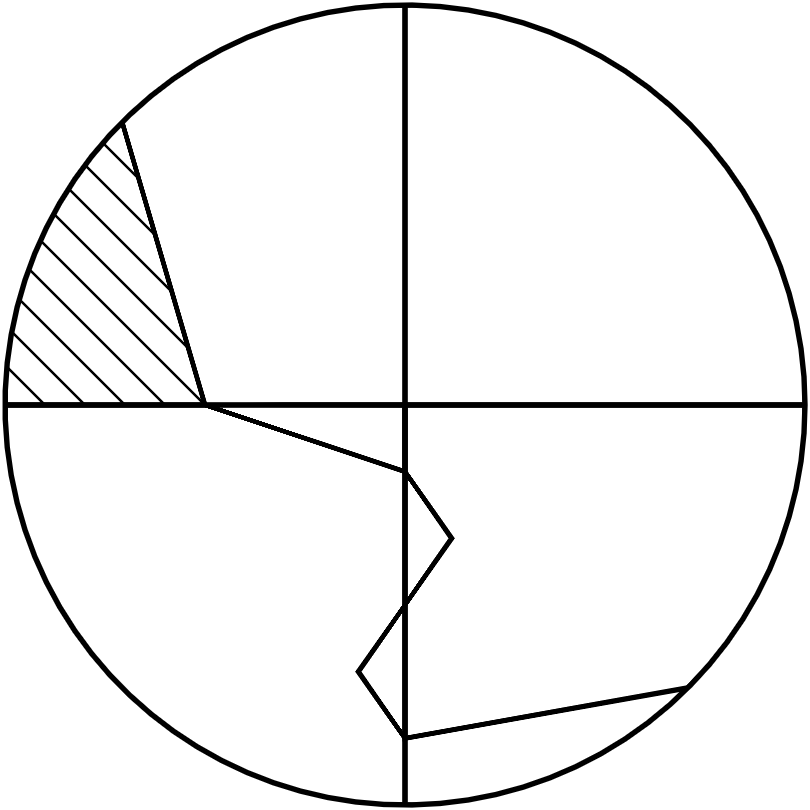}
        \caption{\texttt{rest3-1}}
    \end{subfigure}%
    ~ 
    \begin{subfigure}[t]{0.18\textwidth}
        \centering
        \includegraphics[width=0.9\linewidth]{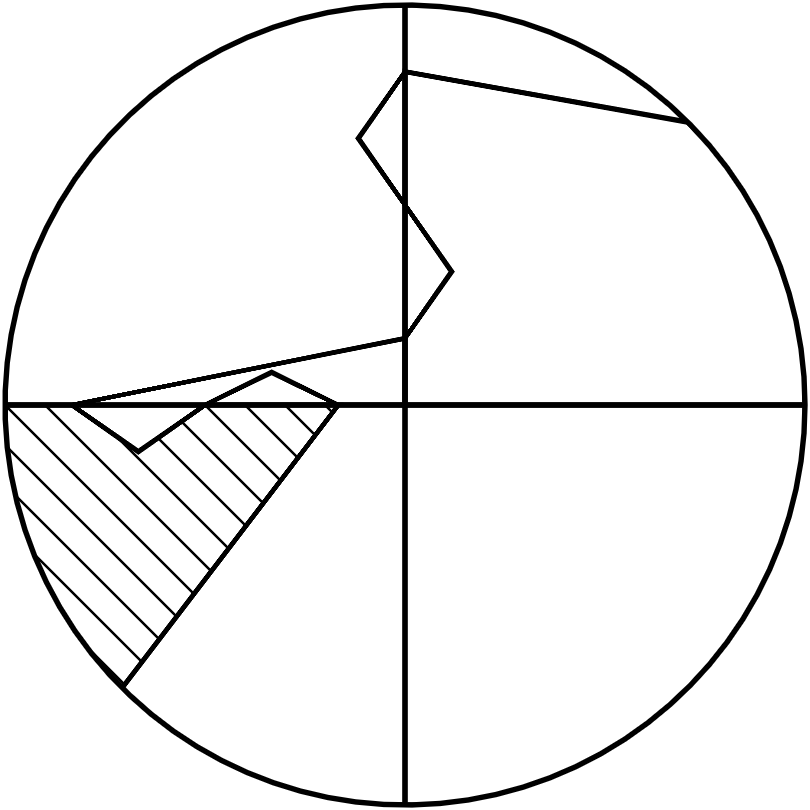}
        \caption{\texttt{rest3-2}}
    \end{subfigure}%
    ~ 
    \begin{subfigure}[t]{0.18\textwidth}
        \centering
        \includegraphics[width=0.9\linewidth]{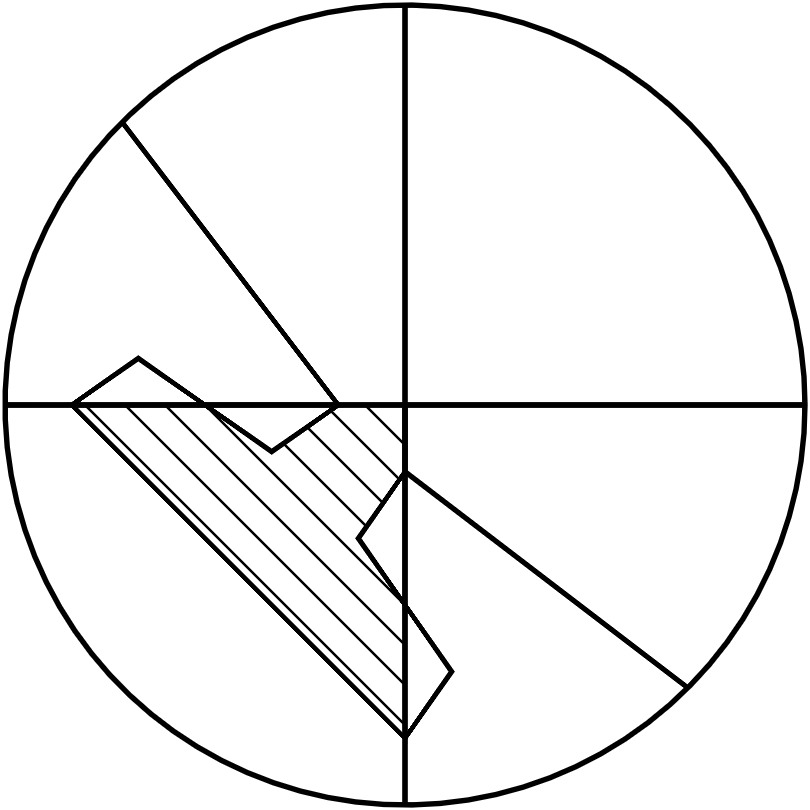}
        \caption{\texttt{rest3-3}}
    \end{subfigure}
    \caption{The three admissible arrangements of three lines and a cubic not obtainable through Viro's patchworking.}
    \label{fig:rest3}
\end{figure*}

All of these three curves are realizable, and explicit equations are given in Table \ref{tab:equations}. The \texttt{Mathematica} \cite{Mathematica} script \texttt{identify.wl} \cite{NWTrepo} certifies that these cubics are nonsingular, intersect $(\calL_x,\calL_y,\calL_z)$ transversely and induce the claimed combinatorial curves. This certification, together with the correctness of the enumeration steps listed before applied to the case of cubic curves, proves Theorem \ref{thm:mainB}: the complete classification of arrangements of three lines and a cubic (Table \ref{tab:xyz3}) is shown in Figure \ref{fig:cubics}.

\begin{table}
    \centering
    \caption{Equations of the three admissible arrangements of three lines and a cubic not obtainable through Viro's patchworking.}
    \begin{tabular}{ c|c }
        Name & Equation \\
        \texttt{rest3-1} & ${\scriptstyle 8x^3 + 22x^2y + 33x y^2 + 7y^3 + 33x^2 z + 31x y z + 34y^2 z + 50x z^2 + 25y z^2 + 3z^3 = 0}$\\
        \texttt{rest3-2} & ${\scriptstyle -504x^3+571x^2 y-190x y^2+15y^3-346x^2 z+242x y z-34y^2 z-79x z^2+25y z^2-6z^3=0}$ \\
        \texttt{rest3-3} & ${\scriptstyle 1001x^3 -530x^2y -400x y^2 +950y^3 +1540x^2 z +1700x y z +1470y^2 z +759x z^2 +738y z^2 +121 z^3 = 0}$ \\
    \end{tabular}
    \label{tab:equations}
\end{table}

\begin{figure}
    \centering
    \includegraphics[width=1\linewidth]{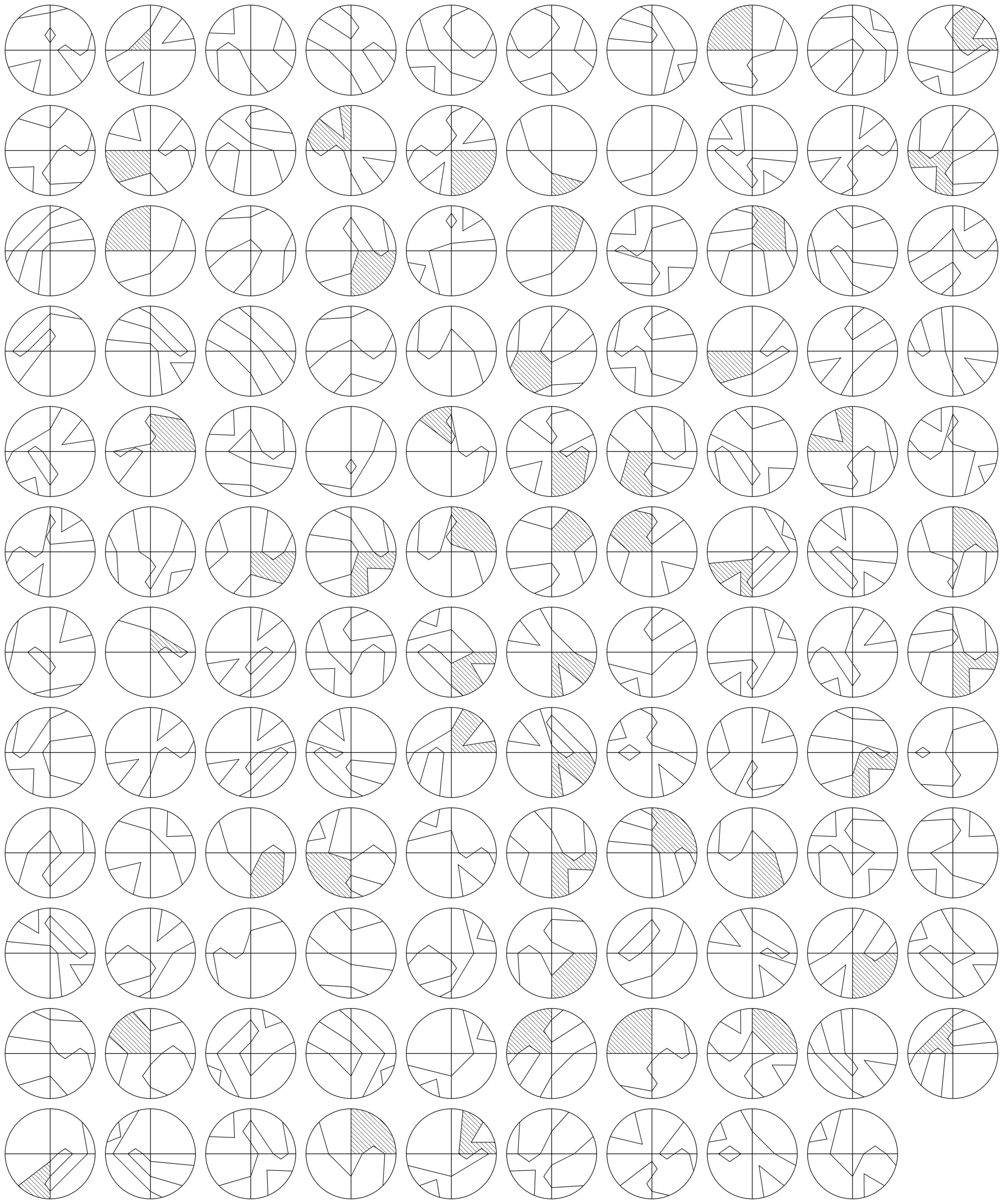}
    \caption{Arrangements up to automorphism of three lines and a cubic.}
    \label{fig:cubics}
\end{figure}

\begin{table}
    \centering
    \caption{\textbf{ \bm{$(\calL_x,\calL_y,\calL_z,\calC)$}, \bm{$\deg(\calC)=3$}}}
    \begin{tabular}{ c|c|c|c }
        & \multicolumn{2}{|c|}{Cubic type} & \multirow{2}{2em}{Total} \\
        & \qtA & \qtB \\
        \hline
        \# Combinatorial curves & 572 & 1500 & 2072 \\
        Up to $i$-automorphism  & 143 &  375  & 518 \\
        Up to automorphism  & 34  &  85  & 119
    \end{tabular}
    \label{tab:xyz3}
\end{table}

We can also specialize this result to configurations of two lines and a cubic (Table \ref{tab:xz3} and Figure \ref{fig:cubics_xz}).

\begin{table}
    \centering
    \caption{\textbf{ \bm{$(\calL_x,\calL_z,\calC)$}, \bm{$\deg(\calC)=3$}}}
    \begin{tabular}{ c|c|c|c }
        & \multicolumn{2}{|c|}{Cubic type} & \multirow{2}{2em}{Total} \\
        & \qtA & \qtB \\
        \hline
        \# Combinatorial curves & 13 & 33 & 46 \\
        Up to $i$-automorphism  & 5 & 11 & 16 \\
        Up to automorphism  & 4 & 8 & 12
    \end{tabular}
    \label{tab:xz3}
\end{table}

\begin{figure}
    \centering
    \includegraphics[width=0.7\linewidth]{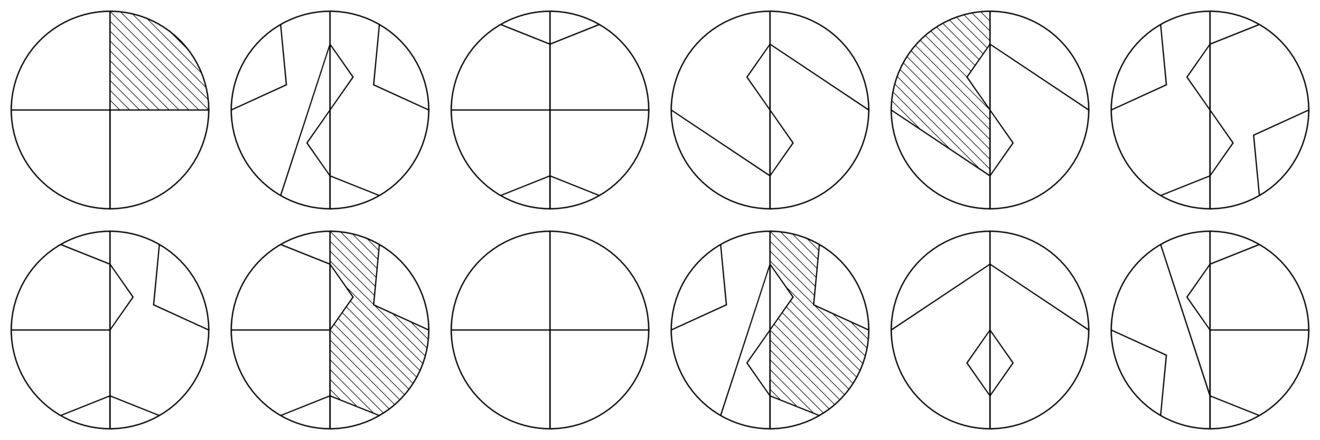}
    \caption{Arrangements up to automorphism of two lines and a cubic.}
    \label{fig:cubics_xz}
\end{figure}

\subsection{Up to three lines and a quartic curve}

Regarding configurations of three lines and a quartic curve, partial results were obtained. In Table \ref{tab:xyz4}, the ``Admissible'' cases are all possible combinatorial curves respecting Bézout at order 0, and that respect Bézout at order 1 after removing all floating ovals. The ``Realized'' cases are those obtained through Viro's patchworking, and after applying two translations as in § \ref{sec:translations}.

\begin{table}
    \centering
    \caption{\textbf{ \bm{$(\calL_x,\calL_y,\calL_z,\calC)$}, \bm{$\deg(\calC)=4$}}}
    \begin{tabular}{ c|c|c|c|c|c|c|c }
        \# Combinatorial curves & \multicolumn{6}{|c|}{Quartic type} & \multirow{3}*{Total} \\
        (Up to $i$-automorphism) & \multirow{2}*{\qtA} & \multirow{2}*{\qtB} & \multirow{2}*{\qtC} & \multirow{2}*{\qtD} & \multirow{2}*{\qtE} & \multirow{2}*{\qtF} \\
        {[Up to automorphism]} & & & & & & & \\
        \hline
        \multirow{3}{9em}{Admissible} & 1 & 13673 & 48208 & 110972 & 210430 & 20631 & 403915 \\
        & (1) & (3434) & (12088) & (27782) & (52684) & (5190) & (101179) \\
        & {[1]} & {[619]} & {[2123]} & {[4834]} & {[9107]} & {[940]} & {[17624]} \\
        \hline
        \multirow{3}{9em}{Realized} & 1 & 13673 & 41560 & 66668 & 72340 & 18687 & 212929 \\
        & (1) & (3434) & (10426) & (16703) & (18142) & (4704) & (53410) \\
        & {[1]} & {[619]} & {[1837]} & {[2921]} & {[3193]} & {[855]} & {[9426]} \\
        \hline
        \multirow{3}{9em}{Difference} & 0 & 0 & 6648 & 44304 & 138090 & 1944 & 190986 \\
        & (0) & (0) & (1662) & (11079) & (34542) & (486) & (47769) \\
        & {[0]} & {[0]} & {[286]} & {[1913]} & {[5914]} & {[85]} & {[8198]}
    \end{tabular}
    \label{tab:xyz4}
\end{table}

We notice that the difference is 0 in the case where the quartic has exactly one oval. Thus, our calculation provides a proof of Theorem \ref{thm:mainC}.

If we specialize our results to floatless configurations (Table \ref{tab:xyz4floatless}), we see that a complete classification is almost obtained: there remain (up to automorphism) 49 cases which are admissible but not realized, shown in Figure \ref{fig:rest_floatless}.

\begin{table}
    \centering
    \caption{\textbf{ \bm{$(\calL_x,\calL_y,\calL_z,\calC)$}, \bm{$\deg(\calC)=4$}, floatless case}}
    \begin{tabular}{ c|c|c|c|c|c|c|c }
        \# Combinatorial curves & \multicolumn{6}{|c|}{Quartic type} & \multirow{3}*{Total} \\
        (Up to $i$-automorphism) & \multirow{2}*{\qtA} & \multirow{2}*{\qtB} & \multirow{2}*{\qtC} & \multirow{2}*{\qtD} & \multirow{2}*{\qtE} & \multirow{2}*{\qtF} \\
        {[Up to automorphism]} & & & & & & & \\
        \hline
        \multirow{3}{9em}{Admissible} & 1 & 13669 & 13494 & 6306 & 1985 & 4815 & 40270 \\
        & (1) & (3433) & (3408) & (1593) & (512) & (1236) & (10183) \\
        & {[1]} & {[618]} & {[619]} & {[305]} & {[105]} & {[235]} & {[1883]} \\
        \hline
        \multirow{3}{9em}{Realized} & 1 & 13669 & 13374 & 5922 & 1529 & 4815 & 39310 \\
        & (1) & (3433) & (3378) & (1497) & (395) & (1236) & (9940) \\
        & {[1]} & {[618]} & {[613]} & {[285]} & {[82]} & {[235]} & {[1834]} \\
        \hline
        \multirow{3}{9em}{Difference} & 0 & 0 & 120 & 384 & 456 & 0 & 960 \\
        & (0) & (0) & (30) & (96) & (117) & (0) & (243) \\
        & {[0]} & {[0]} & {[6]} & {[20]} & {[23]} & {[0]} & {[49]} 
    \end{tabular}
    \label{tab:xyz4floatless}
\end{table}

\begin{figure}
    \centering
    \includegraphics[width=0.7\linewidth]{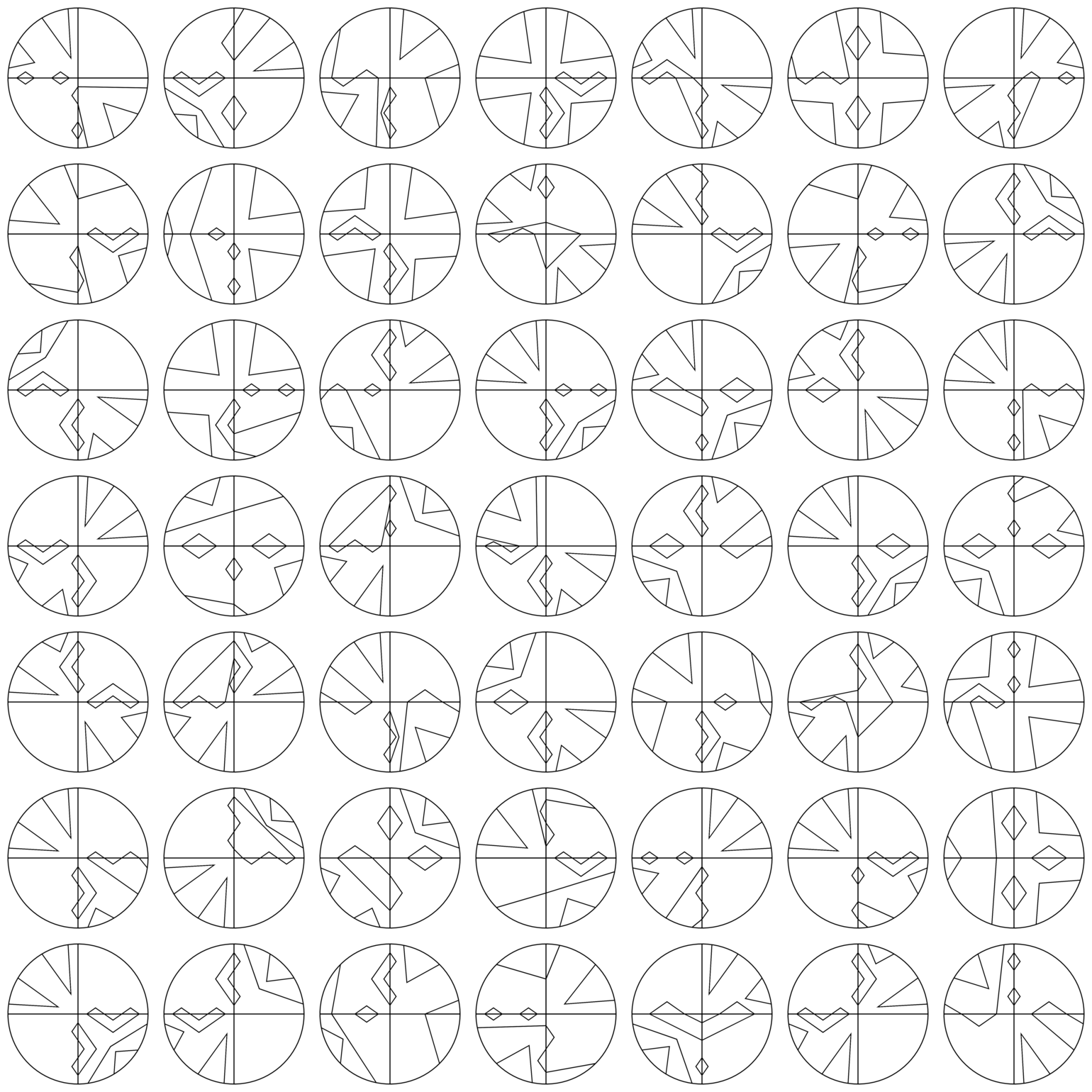}
    \caption{Floatless arrangements of three lines and a quartic which are admissible but not realized.}
    \label{fig:rest_floatless}
\end{figure}

Finally, if we forget one line and look at configurations of two lines and a quartic (Table \ref{tab:xz4}), we are also able to classify all but 39 cases, which are displayed in Figure \ref{fig:quartic_xz_rest}.

\begin{table}
    \centering
    \caption{\textbf{ \bm{$(\calL_x,\calL_z,\calC)$}, \bm{$\deg(\calC)=4$}}}
    \begin{tabular}{ c|c|c|c|c|c|c|c }
        \# Combinatorial curves & \multicolumn{6}{|c|}{Quartic type} & \multirow{3}*{Total} \\
        (Up to $i$-automorphism) & \multirow{2}*{\qtA} & \multirow{2}*{\qtB} & \multirow{2}*{\qtC} & \multirow{2}*{\qtD} & \multirow{2}*{\qtE} & \multirow{2}*{\qtF} \\
        {[Up to automorphism]} & & & & & & & \\
        \hline
        \multirow{3}{9em}{Admissible} & 1 & 97 & 234 & 376 & 519 & 151 & 1378 \\
        & (1) & (32) & (74) & (117) & (161) & (52) & (437) \\
        & {[1]} & {[20]} & {[47]} & {[73]} & {[101]} & {[31]} & {[273]} \\
        \hline
        \multirow{3}{9em}{Realized} & 1 & 97 & 226 & 324 & 389 & 151 & 1188 \\
        & (1) & (32) & (72) & (103) & (125) & (52) & (385) \\
        & {[1]} & {[20]} & {[45]} & {[62]}  & {[75]} & {[31]} & {[234]} \\
        \hline
        \multirow{3}{9em}{Difference} & 0 & 0  &  8 &  52 & 130 &   0 &  190 \\
        & (0) &  (0)  & (2)  & (14)  & (36) &  (0)  & (52) \\
        & {[0]}  & {[0]} &  {[2]} & {[11]} &  {[26]} &  {[0]} &  {[39]}
    \end{tabular}
    \label{tab:xz4}
\end{table}

\begin{figure}
    \centering
    \includegraphics[width=0.9\linewidth]{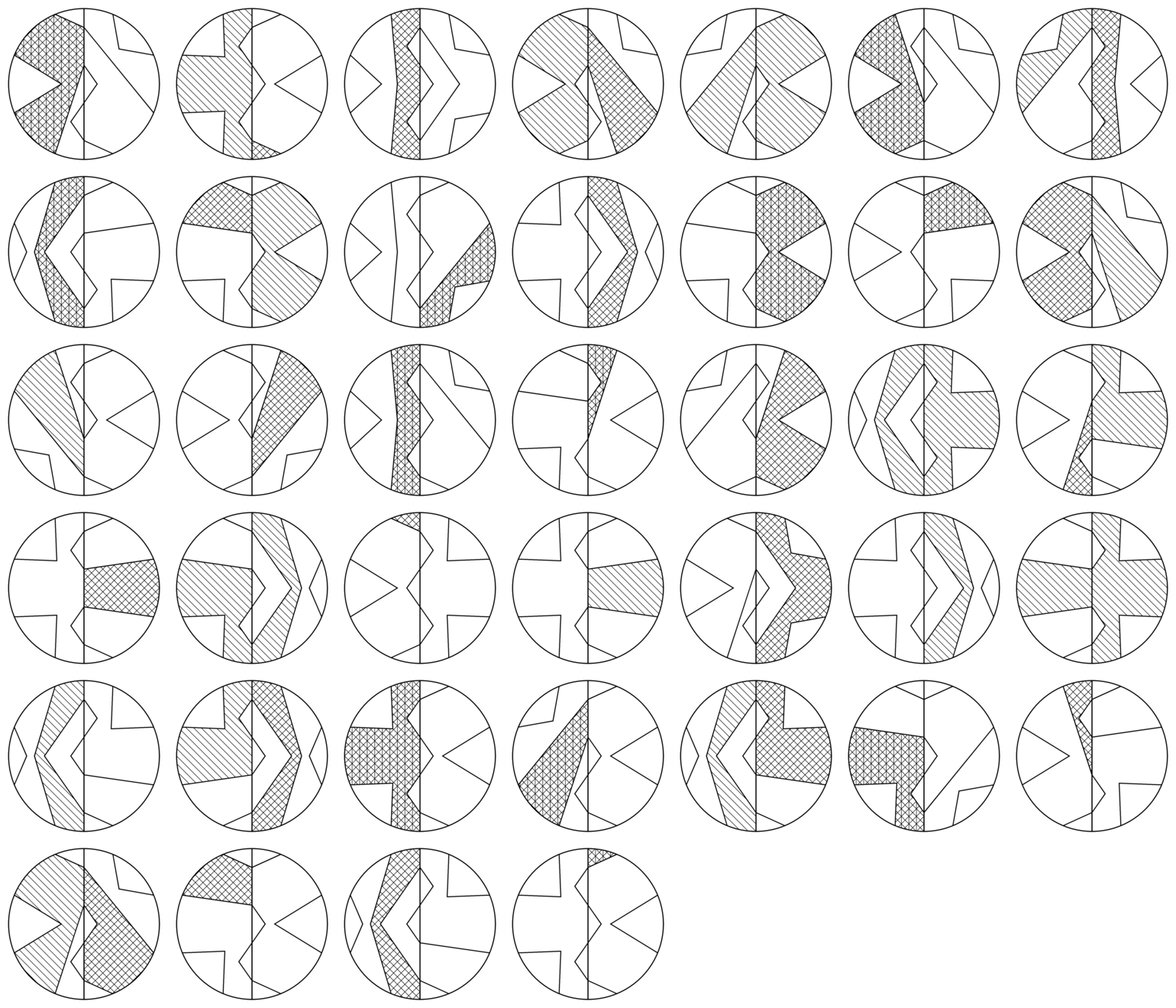}
    \caption{Arrangements of two lines and a quartic curve which are admissible but not realized. The shading represents the number of floating ovals in each ground region.}
    \label{fig:quartic_xz_rest}
\end{figure}

It is important to note that for many of these arrangements, it would not be hard to realize them by hand; rather, we lack a systematic way of producing a realization, or a new obstruction to realizability that would exclude them.

\clearpage

\bibliographystyle{abbrv}
\bibliography{bibliography}

\end{document}